\documentclass[11pt,reqno]{amsart}
\usepackage[left=1.5in, right=1.5in, top=1in, bottom=1in, includefoot]{geometry}
\usepackage{amsmath,amssymb,amscd,amsthm,amsxtra,esint}
\usepackage{colonequals,bm,color}
\usepackage[implicit=true, colorlinks=true, citecolor=red, linkcolor=red, urlcolor=black]{hyperref}
\usepackage{enumerate}   
\usepackage{tikz,pgfplots}

\numberwithin{equation}{section}
\newtheorem{theorem}{Theorem}
\newtheorem{lemma}[theorem]{Lemma}
\newtheorem{proposition}[theorem]{Proposition}
\newtheorem{corollary}[theorem]{Corollary}
\theoremstyle{definition}

\theoremstyle{remark}
\newtheorem{remark}{Remark}
\numberwithin{theorem}{section}

\DeclareMathOperator*{\Tr}{Tr}
\DeclareMathOperator*{\Rea}{Re}
\DeclareMathOperator*{\Ima}{Im}

\newcommand{\N}[1]{ \left\| #1 \right\| }
\def\dH{\dot{\mathfrak{H}}}
\def\R{{\mathbb{R}}}

\def\g{{\gamma}}
\def\E{\mathcal{E}}

\def\L{\mathcal{L}}
\def\S{\mathfrak{S}}

\def\w{{w}}

\begin{document}
\title[Nonlinear Schr\"odinger system of infinitely many particles]{Global existence versus finite time blowup dichotomy for the system of nonlinear Schr\"odinger equations}

\author[Y. Hong]{Younghun Hong}
\email{younghun.hong@yonsei.ac.kr}
\address{Center for Mathematical Analysis \& Computation (CMAC), Yonsei University, Seoul 03722, Korea}

\author[S. Kwon]{Soonsik Kwon}
\email{soonsikk@kaist.edu}
\address{Department of Mathematical Sciences, Korea Advanced Institute of Science and Technology,
291 Daehak-ro Yuseong-gu, Daejeon 34141, South Korea}

\author[H. Yoon]{Haewon Yoon}
\email{hwyoon@ncts.ntu.edu.tw}
\address{National Center for Theoretical Sciences, No. 1 Sec. 4 Roosevelt Rd., National Taiwan University, Taipei, 10617, Taiwan}

\begin{abstract}
We construct an extremizer for the kinetic energy inequality (except the endpoint cases) developing the concentration-compactness technique for operator valued inequality in the formulation of the profile decomposition. Moreover, we investigate the properties of the extremizer, such as the system of Euler-Lagrange equations, regularity and summability. As an application, we study a dynamical consequence of a system of nonlinear Schr\"odinger equations with focusing cubic nonlinearities in three dimension when each wave function is restricted to be orthogonal. Using the critical element of the kinetic energy inequality, we establish a global existence versus finite time blowup dichotomy. This result extends the single particle result of   \cite{{HoRo07}} to infinitely many particles system.
\end{abstract}

\maketitle

\section{Introduction}
\subsection{Setup of the problem}
We consider the infinite system of coupled focusing 3d cubic nonlinear Schr\"odinger equations (CNLS)
\begin{equation}\label{CNLS}
\left\{
\begin{aligned}
i\partial_t\phi_j +\Delta \phi_j +\rho\phi_j&=0,&& j\in\mathbb{N},\\
\phi_j(0)&=\phi_{j,0},
\end{aligned}\right.
\end{equation}
where $\phi_j=\phi_j(t,x):\mathbb{R}\times\mathbb{R}^3\to \mathbb{C}$ and
$$\rho=\sum_{j=1}^\infty|\phi_j|^2$$
is the (total) density function. The following hypotheses are imposed on initial states:
\begin{enumerate}
\item initial states are mutually orthogonal, i.e.,
$$\langle \phi_{j,0}, \phi_{k,0}\rangle_{L^2(\mathbb{R}^3)}=\delta_{jk}\|\phi_{j,0}\|_{L^2(\mathbb{R}^3)}^2;$$
\item $\|\phi_{j,0}\|_{L^2(\mathbb{R}^3)}\to 0$ as $j\to\infty$;
\item initial data $\{\phi_{j,0}\}_{j=1}^\infty$ is normalized so that
$$\max_{j\in\mathbb{N}}\|\phi_{j,0}\|_{L^2(\mathbb{R}^3)}=1.$$ 
\end{enumerate}
Such a system arises in nonlinear optics to describe propagation of spatially incoherent light beams (see \cite{CCMS, MSCC, KSSEC, HLA} for instance). Similar systems with orthogonal initial data appear as the Hartree or the Hartree-Fock systems for the mean-field dynamics of fermions \cite{BGGM, EESY, FK, BPS, BJPSS}.

We define the Sobolev space $\vec{H}^1$ by the Hilbert space of sequences equipped with the norm
$$\left\|\{\phi_j\}_{j=1}^\infty\right\|_{\vec{H}^1}=\left\{\sum_{j=1}^\infty \|\phi_j\|_{H^1}^2\right\}^{1/2}.$$
The system \eqref{CNLS} is locally well-posed in $\vec{H}^1$, and its solution $\vec{\Phi}(t)=\{\phi_j(t)\}_{j=1}^\infty$ obeys the following conservation laws (see Section \ref{sec: local theory}):
\begin{itemize}
\item Orthogonality
\begin{equation}\label{orthogonality of waves}
\langle \phi_j(t), \phi_k(t)\rangle_{L^2(\mathbb{R}^3)}=\delta_{jk}\|\phi_{j,0}\|_{L^2(\mathbb{R}^3)}^2;
\end{equation}
\item Total mass
\begin{equation}\label{total mass}
N(\vec{\Phi})=\sum_{j=1}^\infty \|\phi_j\|_{L^2(\mathbb{R}^3)}^2;
\end{equation}
\item (Total) energy
\begin{equation}\label{energy}
\mathcal{E}(\vec{\Phi}):=\frac{1}{2}\sum_{j=1}^\infty \|\nabla\phi_j\|_{L^2(\mathbb{R}^3)}^2-\frac{1}{4}\sum_{j,k=1}^\infty\int_{\mathbb{R}^3}|\phi_j|^2|\phi_k|^2 dx.
\end{equation}
\end{itemize}
The total mass can be interpreted as the number of particles in that if $\phi_1,\cdots,\phi_N$ are $L^2$-normalized but if $\phi_j=0$ for all $j\geq N+1$, then their total mass equals to the number of particles.

Introducing an operator form by wave functions
\begin{equation}\label{spectral decomposition}
\gamma(t)=\sum_{j=1}^\infty |\phi_j(t)\rangle\langle\phi_j(t)|
\end{equation}
(see Section \ref{sec: Operator spaces} for the bra-ket notation), the system \eqref{CNLS} can be reformulated in a compact form. Indeed it is an operator equation of $ \gamma(t) $ as
\begin{equation}\label{CNLS'}
i\partial_t \gamma=[-\Delta{-}\rho_\gamma, \gamma],
\end{equation}
where $[A,B]=AB-BA$ and
$$\rho_{\gamma(t)}(x)=\gamma(t;x,x)$$
is the density function. Here, the two variable function $\gamma(x,x')$ denotes the integral kernel of the operator $\gamma$, i.e.,
$$(\gamma\phi)(x)=\int_{\mathbb{R}^3} \gamma(x,x')\phi(x')dx'.$$
In terms of wave functions, we have $\rho_{\gamma(t)}(x)= \sum_{j=1}^\infty  \overline{\phi_j(t,x)}\phi_j(t,x)$. Conversely, if $\gamma(t)$ is a compact self-adjoint operator on $L^2(\mathbb{R}^3)$ and it is a solution to the equation \eqref{CNLS'}, then it has a spectral decomposition of the form \eqref{spectral decomposition} and the sequence $\vec{\Phi}(t)=\{\phi_j(t)\}_{j=1}^\infty$ solves the system \eqref{CNLS}. Therefore, the two formulations \eqref{CNLS} and \eqref{CNLS'} are essentially equivalent. Later, either formulation will be taken depending on what is more convenient for our exposition.

By \eqref{spectral decomposition}, the norm of an operator 
$$\|\gamma\|_{\mathfrak{H}^1}:=\textup{Tr}|\sqrt{1-\Delta}\gamma\sqrt{1-\Delta}|,$$
where $\textup{Tr}|\cdot|$ is the trace norm. It corresponds to the $\vec{H}^1$-norm of the sequence $\{\phi_j\}_{j=1}^\infty$. The equation \eqref{CNLS'} is then locally well-posed in the operator space $\mathfrak{H}^1$, where $\mathfrak{H}^1$ is the Banach space of compact self-adjoint operators equipped with the norm $\|\cdot\|_{\mathfrak{H}^1}$. The assumptions on initial data and the conservation laws for the system \eqref{CNLS} can be translated as:
\begin{equation}\label{conservation laws}
\left\{
\begin{aligned}
&\textup{compactness, }\|\gamma_0\|_{\textup{op}}=1&&\textup{(initial data)}\\
&\textup{spectrum of }\gamma&&\textup{(orthogonality)}\\
&\mathcal{E}(\gamma)=\frac{1}{2}\textup{Tr}\sqrt{-\Delta}\gamma\sqrt{-\Delta}-\frac{1}{4}\int_{\mathbb{R}^3}(\rho_\gamma)^2 dx&&\textup{(energy)}\\
&N=\textup{Tr}(\gamma)&&\textup{(total mass)},
\end{aligned}\right.
\end{equation}
where $\|\cdot\|_{\textup{op}}$ denotes the operator norm.

The purpose of this article is to provide a precise description on a global versus blow-up dichotomy for the system (CNLS). In the single-particle case, i.e., for the focusing 3d cubic nonlinear Schr\"odinger equation (NLS)
$$i\partial_t u+\Delta u+|u|^2u=0,$$
Holmer and Roudenko \cite{HoRo07} proved that such a dichotomy is given in terms of the ground state $Q$ for the elliptic equation
$$-\Delta u +u-|u|^2u=0.$$
This result heavily relies on the fact that the best constant for the Gagliardo-Nirenberg inequality
\begin{equation}\label{GN inequality}
\|u\|_{L^4(\mathbb{R}^3)}\leq C_{GN}\|u\|_{L^2(\mathbb{R}^3)}^{1/4}\|\nabla u\|_{L^2(\mathbb{R}^3)}^{3/4}
\end{equation}
is attained at the ground state. Their dichotomy theorem provides the first step to give a sharp criteria for scattering (in other words, to develop a large data scattering theory) \cite{HoRo08, DHR} by the concentration-compactness approach \cite{CKSTT, KM1, KM2}.

The main result of this paper asserts that for the system (CNLS), the \textit{kinetic energy inequality} (see Theorem \ref{theorem: LT inequality} below) plays the role of the Gagliardo-Nirenberg inequality \eqref{GN inequality}. Precisely, we prove existence of an extremizer for the kinetic energy inequality, and then the global versus blow-up dichotomy for the system \eqref{CNLS} (equivalently \eqref{CNLS'}) is obtained using the properties of the critical element. This result is completely analogous to but also extends the single particle case \cite{HoRo07}.

\subsection{Extremizer for kinetic energy inequalities}
Consider a wave function of $ N$-particle system given by a Slater determinant $\frac{1}{\sqrt{N!}}\textup{det}\{\phi_j(x_k)\}_{j,k=1}^N$, where $\{\phi_j\}_{j=1}^N$ is a set of $L^2$-orthonormal functions in $H^1(\mathbb{R}^d)$. The fundamental kinetic energy inequality states that
\begin{equation}\label{LT inequality, even more special case}
\Big\|\sum_{j=1}^N|\phi_j|^2\Big\|_{L^{\frac{d+2}{d}}(\mathbb{R}^d)}^{\frac{d+2}{d}}\lesssim\sum_{j=1}^N\|\nabla \phi_j\|_{L^2(\mathbb{R}^d)}^2.
\end{equation}
This inequality is a generalization of the Gagliardo-Nirenberg inequality
$$\|u\|_{L^{\frac{2(d+2)}{d}}(\mathbb{R}^d)}^{\frac{2(d+2)}{d}}\lesssim\|u\|_{L^2(\mathbb{R}^d)}^{\frac{4}{d}}\|\nabla u\|_{L^2(\mathbb{R}^d)}^2.$$
An important feature of the kinetic energy inequality is that it captures a gain of summability from orthogonality of each state. Indeed, if we estimate the left hand side of \eqref{LT inequality, even more special case} simply by the triangle inequality and the Gagliardo-Nirenberg inequality, then its bound is $O(N^{\frac{d+2}{d}})$, while the bound in \eqref{LT inequality, even more special case} is $O(N)$.

The kinetic energy inequality has been first introduced in the celebrated article by Lieb and Thirring \cite{LiebThirring1}, where the authors established a simpler proof of stability of matter, as well as for a better lower bound constant, in the earlier work by Dyson and Lenard \cite{DysonLenard1, DysonLenard2}. The original proof of the kinetic energy inequality in \cite{LiebThirring1} is done indirectly via its dual formulation, now named the \textit{Lieb-Thirring inequality}, that is, an estimate on the sum of the negative eigenvalues of Schr\"odinger operators. There is a huge literature on these inequalities (see the monograph by Lieb and Seiringer  \cite{LiebSeringer}). Recently, direct proofs of the kinetic energy inequality are given by different approaches, for instance, by Rumin \cite{Rum11}, by Lundholm, Portmann and Solovej \cite{LPS} and by Sabin \cite{Sabin}.

Putting  \eqref{LT inequality, even more special case} into the density matrix formalism as well as without fixing the (endpoint) norm on the left hand side, the kinetic energy inequality can be extended as an inequality for operators as follows. In Section \ref{sec: kinetic energy inequality}, we give a proof of it following Rumin \cite{Rum11} for completeness of the paper.

\begin{theorem}[Kinetic energy inequality{\footnote{In many literatures, \eqref{eq:LT} is also called the Lieb-Thirring inequality as \eqref{eq:LT} is a dual estimate of it. In this article, following the terminology in \cite{Seiringer}, we call \eqref{eq:LT} the kinetic energy inequality merely to be consistent and to avoid the confusion.}}]\label{theorem: LT inequality}
Suppose that 
\begin{equation}
\left\{
\begin{aligned}
\tfrac{d+2}{d}&\leq q<\infty &&\textup{when }d=1,2,\\
\tfrac{d+2}{d}&\leq q\leq\tfrac{d}{d-2} &&\textup{when }d\geq 3,
\end{aligned}
\right.
\end{equation}
and let $\theta=\frac{d}{2q'}\in[\frac{d}{d+2},1]$. There exists a (optimal) constant $C_{KE}>0$ such that if $\gamma$ is bounded, self-adjoint on $L^2(\mathbb{R}^d)$ and $\sqrt{-\Delta}\gamma\sqrt{-\Delta}$ is of trace-class, then
\begin{equation}\label{eq:LT}
	\N{\rho_\gamma}_{L^q(\mathbb{R}^d)}\le C_{KE} \|\gamma\|_{\textup{op}}^{1-\theta}\|\gamma\|_{\dot{\mathfrak{H}}^1}^{\theta},
\end{equation}
where
$$\|\gamma\|_{\dot{\mathfrak{H}}^1}:=\textup{Tr}|\sqrt{-\Delta}\gamma\sqrt{-\Delta}|.$$
\end{theorem}

{We are ready to state our first main theorem. We prove the existence of an extremizer for the kinetic energy inequality, derive the Euler-Lagrange system and its qualitative properties. We collect statements regarding to the exremizer in the following theorem. }

\begin{theorem}[Extremizer for the kinetic energy inequality \eqref{eq:LT}]\label{main theorem}
Suppose that
\begin{equation}\label{admissible q}
\left\{
\begin{aligned}
\tfrac{d+2}{d}&< q<\infty &&\textup{when }d=1,2,\\
\tfrac{d+2}{d}&<q<\tfrac{d}{d-2} &&\textup{when }d\geq 3,
\end{aligned}
\right.
\end{equation}
and let $\theta=\frac{d}{2q'}\in (\frac{d}{d+2},1)$. Then, the following hold.
\begin{enumerate}
\item (Existence) There exists an extremizer $\mathcal{Q}$ for the kinetic energy inequality \eqref{eq:LT}, that is,
\begin{equation}\label{sharp kinetic energy inequality}
\|\rho_{\mathcal{Q}}\|_{L^q(\mathbb{R}^d)}=C_{KE}\|\mathcal{Q}\|_{\textup{op}}^{1-\theta}\|\mathcal{Q}\|_{\dot{\mathfrak{H}}^1}^\theta,
\end{equation}
such that $\mathcal{Q}\geq0$, $\|\mathcal{Q}\|_{\textup{op}}=1$ and $\|\mathcal{Q}\|_{\dot{\mathfrak{H}}^1}=\theta\|\rho_{\mathcal{Q}}\|_{L^q(\mathbb{R}^d)}^q$.
\item (Structure) We denote by $\{-\mu_j\}_{j=1}^{J_-}$ the set of negative eigenvalues (counting multiplicities) for the Schr\"odinger operator $(-\Delta-\rho_{\mathcal{Q}}^{q-1})$ with the ordering
$$-\mu_1\leq -\mu_2\leq-\mu_3\leq\cdots <0,$$
and let $\phi_j^-$ be the $L^2$-normalized eigenfunction corresponding to the eigenvalue $-\mu_j$. Then, either 
$$\mathcal{Q}=\sum_{j=1}^{J_-}|\phi_j^-\rangle\langle\phi_j^-|$$
or there exists an $L^2$-orthogonal set $\{\phi_k^0\}_{k=1}^{K_0}\subset\textup{Ker}(-\Delta-\rho_{\mathcal{Q}}^{q-1})$ such that 
\begin{equation}\label{spectral decomposition for Q}
\mathcal{Q}=\sum_{j=1}^{J_-}|\phi_j^-\rangle\langle\phi_j^-|+\sum_{k=1}^{K_0}|\phi_k^0\rangle\langle\phi_k^0|.
\end{equation}
Here, $J_-, K_0 \in \mathbb{N} \cup \{\infty\}$.
\item (Euler-Lagrange equation) Each $\phi_j^-$ (resp., $\phi_k^0$ if it exists) is an $H^1$-weak solution to
\begin{equation}\label{EL}
(-\Delta-\rho_{\mathcal{Q}}^{q-1})\phi_j^-=-\mu_j\phi_j^-\ \left(\textup{resp., }(-\Delta-\rho_{\mathcal{Q}}^{q-1})\phi_k^0=0\right).
\end{equation}
\item (Regularity)
$$\textup{Tr}(-\Delta)\mathcal{Q}(-\Delta)<\infty.$$
\item (Pohozaev identities)
\begin{equation}\label{Po}
\|\mathcal{Q}\|_{\dot{\mathfrak{H}}^1}=\frac{\theta}{1-\theta}\sum_{j=1}^{J_-} \mu_j;\quad\|\rho_{\mathcal{Q}}\|_{L^q(\mathbb{R}^d)}^q=\frac{1}{1-\theta}\sum_{j=1}^{J_-} \mu_j.
\end{equation}
\item (Summability) If we further assume that $d\geq 3$ and $q>\frac{d^2+2d+4}{d^2}$, then $J_-$ is finite and $\textup{Tr}\mathcal{Q}<\infty$.
\end{enumerate}
\end{theorem}

\begin{remark}\label{remark 3}
{In Theorem \ref{main theorem}, the endpoint cases $q=\frac{d+2}{d}$ or $\frac{d}{d-2}$) are missing. This is due to that we have used the interpolation in the proof of the profile decomposition (Theorem \ref{profile decomposition}). This kind of defect appear naturally in various setting. }
\end{remark}

We prove existence of an extremizer developing the concentration-compactness principle adapted to the class of operators admissible to the kinetic energy inequality \eqref{eq:LT}. The concentration-compactness principle, describing all possible failures of compactness, was introduced by Lions \cite{Lions1,Lions2} for bounded functions, and it has been a fundamental tool in the area of calculus of variation. This principle is also known as a bubble decomposition in the study of minimal surfaces \cite{SaUl,BrCo}. For dispersive and wave equations, the method was intensively used by many authors in the study of critical equations, for proving blowup, mass concentration or scattering \cite{HmKe05,BG,KM1,KM2}. In \cite{Lio87}, the concentration-compactness principle for finitely many orthonormal functions was developed to study the Hartree-Fock  system of $N$ particles. In \cite{DFL, LL, ADS}, the method is adapted to the trace class operators to solve variational problems for operators. In this paper, the concentration-compactness principle is not only extended to the non-traceable class $\mathfrak{B}^1$ of operators (see \eqref{B definition}), but it is also formulated as a profile decomposition (for possible applications in the future work).

The kinetic energy inequality \eqref{eq:LT} is not compact due to translation invariance, which is a noncompact symmetry. Thus, taking up to a subsequence, a bounded sequence in $\mathfrak{B}^1$ can be expressed as a sum of orthogonally translating profiles and reminder term which is convergent in potential energy. Due to the profile decomposition, one actually see that the translation symmetry is responsible for the noncompactness of the embedding. Then, we combine this with a binding inequality of potential energy ($=L^q$-norm) to obtain that an extremizing sequence has the Palais-Smale condition. In our setting, since we handle bounded operators (although their kinetic terms are traceable), we need to make suitable modifications and also encounter several difficulties. For instance, contrary to the usual profile decompositions (see \cite{HmKe05}, for instance), asymptotic orthogonality of the profiles cannot be seen in the operator norm (see Theorem \ref{profile decomposition} (3) and Remark \ref{remark on the profile decomposition} (1)). Moreover, to get the desired estimate for the remainder (Theorem \ref{profile decomposition} (5)), we have to keep each profile is self-adjoint and non-negative. These will be taken in account in the proof of the profile decomposition.

Next, we derive the Euler-Lagrange equation for the extremizer when we express the operator as orthogonal states. In this step, we use the argument introduced in \cite{HLS, DFL, LL, ADS, FLLS13}. Indeed, we can transfer the variational problem associated to the kinetic energy to the minimization problem for a Schr\"odinger operator with a potential. As a consequence, we result in a system of \textit{self-consistent equations}, that is, the Euler-Lagrange equations. Then, using this elliptic system, we derive some qualitative properties of the extremizer.

\subsection{Global existence versus finite time blow-up dichotomy}
From now on, we fix $d=3$ and $q=2$, and consider the kinetic energy inequality
\begin{equation}\label{eq:LT PDE application}
\N{\rho_\gamma}_{L^2(\mathbb{R}^3)}\le C_{KE} \|\gamma\|_{\textup{op}}^{1/4}\|\gamma\|_{\dot{\mathfrak{H}}^1}^{3/4}.
\end{equation}
As an application of Theorem \ref{main theorem}, one can characterize the geometry of the dense subset
$$\tilde{\mathcal{S}}:=\Big\{\gamma\in\mathcal{L}:\ \|\gamma\|_{\textup{op}}=1\textup{ and }\|\gamma\|_{\dot{\mathfrak{H}}^1}<\infty\Big\}$$
of the unit sphere in the Banach space $(\mathcal{L}, \|\cdot\|_{\textup{op}})$ of bounded operators. Indeed, it follows from the sharp kinetic energy inequality \eqref{sharp kinetic energy inequality} and the Pohozaev identities \eqref{Po} that there is a forbidden region
$$\left\{(\alpha, y): y\leq f(\alpha):=\frac{1}{2}\alpha^2-\frac{1}{3\|\mathcal{Q}\|_{\dot{\mathfrak{H}}^1}^{1/2}}\alpha^3\right\},$$
on the $\alpha$-$y$ plane with $\alpha=\|\gamma\|_{\dot{\mathfrak{H}}^1}^{1/2}$ and $y=\mathcal{E}(\gamma)$:
Thus, if $\mathcal{E}(\gamma)<\mathcal{E}(Q)$, then either $\|\gamma\|_{\dot{\mathfrak{H}}^1}<\|\mathcal{Q}\|_{\dot{\mathfrak{H}}^1}$ or $\|\gamma\|_{\dot{\mathfrak{H}}^1}>\|\mathcal{Q}\|_{\dot{\mathfrak{H}}^1}$.

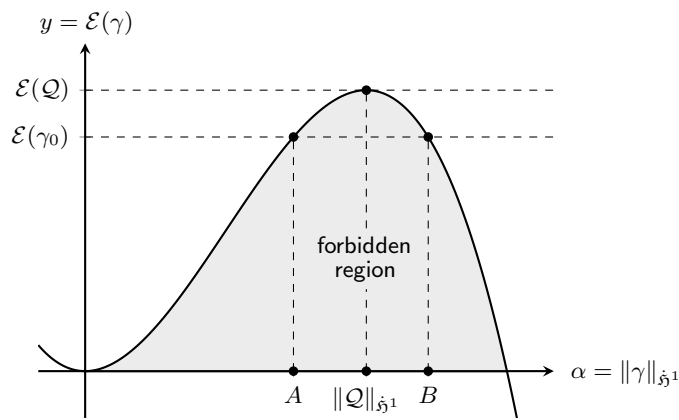
\begin{figure}[h!]
\centering
\def\Cubic(#1){(#1)^2/2-(#1)^3/18}
\begin{tikzpicture}
	\begin{axis}
	[xmin=-1, xmax=10, ymin=-1, ymax=7,
        axis lines = middle, axis line style = thick,  unit vector ratio = 1 1 1,
        xtick = {4.4451, 6, 7.3206},
        xticklabels = {\footnotesize$A$, \footnotesize$\|\mathcal{Q}\|_{\dH^1}$, \footnotesize$B$},
        ytick = {5,6},
        yticklabels = {\footnotesize$\E(\g_0)$, \footnotesize$\E(\mathcal{Q})$}]
        \addplot[domain=0:9, samples=200, fill=gray!15, draw=none] {\Cubic(x)};
        \addplot[mark=*,mark options={scale=0.8, solid}, dashed] coordinates {(4.4451, 0) (4.4451, 5)};
        \addplot[dashed] coordinates {(6, 0) (6, 1.75)};
        \addplot[dashed] coordinates {(6, 3.05) (6, 6)};
        \addplot[mark=*, mark options={scale=0.8, solid}] coordinates {(6, 0)};
        \addplot[mark=*, mark options={scale=0.8, solid}] coordinates {(6, 6)};
        \addplot[mark=*, mark options={scale=0.8, solid}, dashed] coordinates {(7.3206, 0) (7.3206, 5)};    
        \addplot[domain=0:10, dashed] {5};
        \addplot[domain=0:10, dashed] {6};
        \addplot[domain=-1:10, samples=200, mark=none, thick] {\Cubic(x)};
        \addplot[thick] coordinates {(0, 0) (9, 0)};
	\end{axis}
	\draw(5.1, 2.3)node[left]{\footnotesize\textsf{forbidden}};
	\draw(4.9, 1.95)node[left]{\footnotesize\textsf{region}};
	\draw(8.7, 0.6)node[left]{\footnotesize$\alpha=\|\g\|_{\dH^1}$};
	\draw(1.4, 5.25)node[left]{\footnotesize$y=\E(\g)$};
\end{tikzpicture}
\caption{Forbidden region on the $\alpha$-$y$ plane}\label{fig1}
\end{figure}

Coming back to the dynamical PDE problem, combined with the conservation law $\mathcal{E}$ in the above observation, we can prove the global existence versus finite time blow-up dichotomy.
\begin{theorem}[Global existence versus finite time blowup dichotomy for CNLS]\label{Dichotomy}
Let $\mathcal{Q}$ be an extremizer for the kinetic energy inequality \eqref{eq:LT PDE application} normalized so that $\|\mathcal{Q}\|_{\textup{op}}=1$. Suppose that $\gamma_0\in\mathfrak{H}^1$, $\gamma_0\geq 0$, $\|\gamma_0\|_{\textup{op}}=1$ and
$$\E(\gamma_0)<\E(\mathcal{Q}).$$
Let $\gamma\in C_t(I_{max}; \mathfrak{H}^1)$ be a solution to \eqref{CNLS'} with initial data $\gamma_0$, where $I_{max}$ is the maximal interval of existence. Then, the following hold.
\begin{enumerate}
\item If $\|\gamma_0\|_{\dot{\mathfrak{H}}^1}<\|\mathcal{Q}\|_{\dot{\mathfrak{H}}^1}$, then $\gamma(t)$ exists in $\mathfrak{H}^1$ globally in time. Moreover, $\|\gamma(t)\|_{\dot{\mathfrak{H}}^1}<\|\mathcal{Q}\|_{\dot{\mathfrak{H}}^1}$ on $I_{max}=\R$.
\item If $\|\gamma_0\|_{\dot{\mathfrak{H}}^1}>\|\mathcal{Q}\|_{\dot{\mathfrak{H}}^1}$, then $\|\gamma(t)\|_{\dot{\mathfrak{H}}^1}>\|\mathcal{Q}\|_{\dot{\mathfrak{H}}^1}$ on $I_{max}$. If we further assume finite variance of initial data, i.e., $\int_{\R^3}|x|^2\rho_{\g_0}dx<\infty$, then $I_{max}$ is finite and thus $\|\gamma(t)\|_{\dot{\mathfrak{H}}^1}$ blows up in finite time.
\end{enumerate}
\end{theorem} 

\begin{remark}
In order to see the role of the kinetic energy inequality for the PDE problem, we consider the system \eqref{CNLS} of $N$ orthonormal functions $\{\phi_j(t)\}_{j=1}^N$. If we ignore orthogonality and use the triangle and the Gagliardo-Nirenberg inequalities, then one can find a forbidden region on the $\alpha$-$y$ plane as well as a dichotomy as above. However, as $N$ increases to infinity, this forbidden region shrinks to $(0,0)$ and thus much less will be known from the picture, while the dichotomy in Theorem \ref{Dichotomy} is valid uniformly in the number of orthonormal functions.
\end{remark}

\begin{remark}
In the dynamical result, we do not pursue to present full general results in various settings. For simplicity of presentation, we just illustrate a typical result which may be (most) interesting, that is , the case of three dimensional cubic nonlinearity case. As like in \cite{HoRo07}, we believe that the main results of this paper is extended to any model which is $L^2$-subcritical but $H^1$-subcritical in scaling, including the NLS-type
$$i\partial_t \gamma=[-\Delta+(\rho_\gamma)^a, \gamma]$$
and the nonlinear Hartree type
$$i\partial_t \gamma=[-\Delta+\tfrac{1}{|x|^a}*\rho_\gamma, \gamma].$$
\end{remark}

\subsection{Organization of the paper}
In Section 2, we prepare the proof of the main theorems introducing operators spaces and their properties as well as giving a proof of the kinetic energy inequality. In Section 3, we develop the profile decomposition for operators (Theorem \ref{profile decomposition}), which is the main analytic tool of this paper. In Section 4, using the profile decomposition, we construct an extremizer for the kinetic energy inequality (Theorem \ref{main theorem} (1)). In Section 5 and 6, we prove the properties of an extremizer including structures, the Euler-Lagrange equation, regularity and summability (Theorem \ref{main theorem} (2)-(6)). Finally, in Section 7, coming back to the PDE problem, we establish a global versus blow-up dichotomy for the infinitely coupled NLS \eqref{CNLS} (Theorem \ref{Dichotomy}).

\subsection{Notations}\label{notations}
For notational convenience, if there is no confusion, we omit the range of indices in sequences and sums. Precisely, we simply write 
$$\{a_j\}\quad(\textup{resp. }\sum a_j) \qquad \textup{for} \qquad  \{a_j\}_{j=1}^J \quad (\textup{resp.  } \sum_{j=1}^J a_j) $$
for  with $J\in \mathbb{N}\cup\{\infty\}$. Similarly, we denote 
$$a_n\to a$$
if $n$ is clearly sent to infinity.

\subsection{Acknowledgement}
Y.H. was partially supported by NRF of Korea (NRF-2017R1C1B1008215). S.K. was partially supported by NRF of Korea (NRF-2015R1D1A1A01058832). Y.H. would like to thank Prof. Mathieu Lewin for valuable comments and discussions while they visited The University of Texas at Austin in March 2017.

\vspace{10pt}

\section{Preliminaries}

In Section \ref{sec: Operator spaces}, we introduce the operator spaces that will be used in this paper, most importantly the Sobolev-type space $\mathfrak{B}^\alpha$ associated with the kinetic energy inequality. In Section \ref{properties of the space}, we give  some important properties of the space $\mathfrak{B}^\alpha$, focusing on density functions. Then, in Section \ref{sec: kinetic energy inequality}, we prove the kinetic energy inequality. Finally, in Section \ref{sec: Low frequency approximation}, as an application, we show the mid-frequency approximation for density functions. In fact, most materials in Section \ref{sec: Operator spaces}-\ref{sec: kinetic energy inequality} have been addressed in \cite{FLLS13} in a more complicated setting, but we present them here for completeness' sake. In the meantime, we amend and simplify some of the materials, since we are considering a simpler setup.

\subsection{Operator spaces}\label{sec: Operator spaces}
We denote by
$$\mathcal{L}=\mathcal{L}(L^2(\R^d))$$
the Banach space of bounded linear operators on $L^2(\R^d)$ with the operator norm $\|\cdot\|_{\textup{op}}$. For $1\leq p\leq\infty$, we define the \emph{Schatten $p$-class} $\S^p$ by the Banach space of compact operators on $L^2(\R^d)$ equipped with the norm
$$\|\gamma\|_{\mathfrak{S}^p}:=\left\{\begin{aligned}
&(\Tr |\gamma|^p)^{\frac{1}{p}}&&\textup{if }1\leq p<\infty,\\
&\|\gamma\|_{\textup{op}}&&\textup{if }p=\infty.
\end{aligned}\right.$$
We remark that $ \mathfrak{S}^1$ is the trace-class, $\mathfrak{S}^2 $ is the Hilbert-Schmidt class, and $\mathfrak{S}^\infty$ is the set of all compact operators.

We recall the physicists' bra-ket notation: given $\phi, \psi\in L^2(\mathbb{R}^d)$, $|\phi\rangle\langle\psi|$ denotes the integral operator with the kernel $\phi(x)\overline{\psi(x')}$, i.e., 
$$f(x)\mapsto\Big(|\phi\rangle\langle\psi|f\rangle\Big)(x) =  \phi(x)\int_{\mathbb{R}^d} \overline{\psi(x')}f(x') dx,$$
while 
$$\langle \phi|\psi\rangle=\int_{\mathbb{R}^d}\overline{\phi(x')}\psi(x')dx'.$$
With this bra-ket notation, a self-adjoint operator $\gamma$ in $\mathfrak{S}^p$ can be written as an eigenfunction expansion
	\[\gamma=\sum\lambda_j|\phi_j\rangle\langle \phi_j|,\]
where $\{\phi_j\}$ is an orthonormal set in $L^2(\mathbb{R}^d)$. Note that $\{\lambda_j\}$ is the set of eigenvalues of the operator $\gamma$. Thus, the Schatten $p$-norm of a self-adjoint operator is simply the $\ell^p$-norm of its eigenvalues,
$$\|\gamma\|_{\mathfrak{S}^p}=\left\|\{\lambda_j\}\right\|_{\ell^p}.$$

Now, we introduce the Sobolev spaces for self-adjoint operators that we will use. Given $\alpha\in\R$, we define the Sobolev space $\dot{\mathfrak{H}}^\alpha$ (resp. $\mathfrak{H}^\alpha$) by the collection of self-adjoint operators equipped with the norm
$$\|\gamma\|_{\dot{\mathfrak{H}}^\alpha}:=\N{|\nabla|^\alpha \gamma|\nabla|^\alpha}_{\mathfrak{S}^1}\quad\left(\textup{resp. }\|\gamma\|_{\mathfrak{H}^\alpha}:=\N{\langle\nabla\rangle^\alpha \gamma\langle\nabla\rangle^\alpha}_{\mathfrak{S}^1}\right)$$
where $|\nabla|^\alpha$ (resp. $\langle\nabla\rangle^\alpha$) is the Fourier multiplier operator with the symbol $|\xi|^{\alpha}$  (resp. $\langle\xi\rangle^\alpha=(1+|\xi|^2)^{\frac{\alpha}{2}}$). We also define the inhomogeneous Sobolev space $\mathfrak{B}^\alpha$ by the collection of bounded self-adjoint operators with the norm 
\begin{equation}\label{B definition}
\|\gamma\|_{\mathfrak{B}^\alpha}:=\|\gamma\|_{\textup{op}}+\|\gamma\|_{\dot{\mathfrak{H}}^\alpha}=\|\gamma\|_{\textup{op}}+\N{|\nabla|^\alpha \gamma|\nabla|^\alpha}_{\mathfrak{S}^1}, \end{equation}
which fits into the kinetic energy inequality. We remark that for technical conveniences (e.g. to use eigenfunction expansions), operators in $\mathfrak{B}^\alpha$, as well as those in $\dot{\mathfrak{H}}^\alpha$, are restricted to be self-adjoint, while Schatten-class operators are not necessarily self-adjoint.

We employ the weak-$\ast$ topology on the operator space $\mathfrak{B}^\alpha$ induced by the  intersection of the relative topologies of two weak-$\ast$ topologies on $\mathcal{L}$ and $\dH^\alpha$. In other words, we say that $\g_n\rightharpoonup^*\g\quad\textup{in }\mathfrak{B}^\alpha$ if $\g_n\rightharpoonup^* \g$ in $\mathcal{L}$ as well as in $\dot{\mathfrak{H}}^\alpha$, that is,
\begin{align*}
\left\{\begin{aligned}
\textup{Tr}(\g_n K)&\to\Tr(\g K)&&\forall K\in\mathfrak{S}^1,\\
\textup{Tr}(|\nabla|^\alpha \g_n |\nabla|^\alpha K')&\to\Tr(|\nabla|^\alpha\g|\nabla|^\alpha K')&&\forall K'\in\mathfrak{S}^\infty
\end{aligned}\right.
\end{align*}
by the duality relations $(\mathcal{L},\N{\cdot}_{\textup{op}})=(\S^1,\N{\cdot}_{\S^1})^\ast$ and $(\S^1,\N{\cdot}_{\S^1})=(\S^\infty,\N{\cdot}_{\textup{op}})^\ast$ (see \cite[Theorem VI.26]{ReSi}).

\subsection{Properties of the operator space $\mathfrak{B}^\alpha$}\label{properties of the space}
For non-negative operators in $\mathfrak{B}^\alpha$, density functions are properly defined in the following sense.
\begin{lemma}\label{lem:densitylocL1}
Let $\alpha>0$. If $\gamma\in\mathfrak{B}^\alpha$ and $\gamma\geq0$, then its density function $\rho_\gamma$ is locally in $L^1(\mathbb{R}^d)$. Moreover, if $\g_n\rightharpoonup^\ast0$ in $\mathfrak{B}^\alpha$ and $\gamma_n\geq0$, then $\rho_{\g_n}\to0$ in $L^1_{loc}(\R^d)$.
\end{lemma}

\begin{proof}
It suffices to show that $\textup{Tr}(\chi\gamma\chi)<\infty$ and $\textup{Tr}(\chi\gamma_n\chi)\to 0$ for any compactly supported smooth cut-off $\chi \in C^\infty_0(\R^d)$, because by non-negativity,
$$\textup{Tr}(\chi\gamma\chi)=\int_{\mathbb{R}^d}\rho_{\chi \gamma\chi}dx=\int_{\mathbb{R}^d}\chi^2\rho_{\gamma}dx=\|\chi^2\rho_{\gamma}\|_{L^1}$$
and similarly, $\textup{Tr}(\chi\gamma_n\chi)=\cdots=\|\chi^2\rho_{\gamma_n}\|_{L^1}$.

Given $\chi \in C^\infty_0(\R^d)$, we decompose 
\begin{align*}
\textup{Tr}(\chi\gamma\chi)&=\textup{Tr}(\chi\Pi_1^-\gamma\chi)+\textup{Tr}(\chi\Pi_1^+\gamma\Pi_1^-\chi)+\textup{Tr}(\chi\Pi_1^+\gamma\Pi_1^+\chi)\\
&=\textup{Tr}\left(\chi\Pi_1^-\g\chi\right)+\textup{Tr}\left(\chi\Pi_1^+\g\Pi_1^-\chi\right)\\
&\quad+\textup{Tr}\left((\chi\Pi_1^+|\nabla|^{-\alpha})(|\nabla|^\alpha\g|\nabla|^{\alpha})(|\nabla|^{-\alpha}\Pi_1^+\chi)\right),
\end{align*}
where $\Pi_1^-\colonequals\mathbf{1}_{(-\Delta)\leq 1}$ and $\Pi_1^+:=1-\Pi_1^-$. Note that $\chi\Pi_1^+|\nabla|^{-\alpha}\in\mathfrak{S}^p$ for some $p\geq 2$ (so $\in\mathfrak{S}^\infty$) by the Kato-Seiler-Simon inequality \cite[Theorem 4.1]{Simon},
	\begin{equation}\label{Holder Schatten}
	\|f(x)g(|\nabla|)\|_{\mathfrak{S}^p}\leq (2\pi)^{-\frac{d}{p}}\|f\|_{L^p}\|g\|_{L^p},	\quad \forall p\geq 2,
	\end{equation}
and that $\chi\Pi_1^-$ is of trace-class by the Birman-Solomjak inequality \cite[Theorem 4.5]{Simon},
$$\|f(x)g(|\nabla|)\|_{\mathfrak{S}^1}\lesssim \|f\|_{2;1}\|g\|_{2;1},$$
where $\|f\|_{2;1}=\sum_{z\in\mathbb{Z}^d}\|f\|_{L^2(C_z)}$ and $C_z$ is the unit cube centered at $z$. Therefore, we show that
\begin{equation}\label{proof of lem:densitylocL1}
\begin{aligned}
\textup{Tr}(\chi\gamma\chi)&\leq\|\chi\Pi_1^-\|_{\mathfrak{S}^1}\|\g\|_{\textup{op}}\|\chi\|_{\textup{op}}+\|\chi\Pi_1^+\|_{\textup{op}}\|\g\|_{\textup{op}}\|\Pi_1^-\chi\|_{\mathfrak{S}^1}\\
&\quad+\|\chi\Pi_1^+|\nabla|^{-\alpha}\|_{\textup{op}}\||\nabla|^\alpha\g|\nabla|^{\alpha}\|_{\mathfrak{S}^1}\||\nabla|^{-\alpha}\Pi_1^+\chi\|_{\textup{op}}\\
&<\infty.
\end{aligned}
\end{equation}
Next, to show that $\textup{Tr}(\chi\gamma_n\chi)\to 0$, we decompose
\begin{align*}
\textup{Tr}(\chi\gamma_n\chi)&=\textup{Tr}\left(\chi\Pi_1^-\g_n\chi\right)+\textup{Tr}\left(\chi\Pi_1^+\g_n\Pi_1^-\chi\right)\\
&\quad+\textup{Tr}\left((\chi\Pi_1^+|\nabla|^{-\alpha})(|\nabla|^\alpha\g_n|\nabla|^{\alpha})(|\nabla|^{-\alpha}\Pi_1^+\chi)\right).
\end{align*}
Since each operator on the right hand side is of trace-class by the estimates in \eqref{proof of lem:densitylocL1}, using cyclicity of the trace, we may write
\begin{align*}
\textup{Tr}(\chi\gamma_n\chi)&=\textup{Tr}\left(\g_n(\Pi_1^-\chi^2)\right)+\textup{Tr}\left(\g_n(\Pi_1^-\chi^2\Pi_1^+)\right)\\
&\quad+\textup{Tr}\left((|\nabla|^\alpha\g_n|\nabla|^{\alpha})(|\nabla|^{-\alpha}\Pi_1^+\chi^2\Pi_1^+|\nabla|^{-\alpha})\right).
\end{align*}
Here, on the right hand side, $\gamma_n$ is tested against trace-class operators but $|\nabla|^\alpha\g_n|\nabla|^{\alpha}$ is tested against a compact operator. Therefore, from the weak-$\ast$ convergence $\g_n\rightharpoonup^\ast0$ in $\mathfrak{B}^\alpha$, we conclude that $\textup{Tr}(\chi\gamma_n\chi)\to 0$.
\end{proof}

Next, we prove that a non-negative operator in $\mathcal{B}^\alpha$ can be approximated by a compactly supported smooth finite-rank operator.

\begin{lemma}\label{density} Let $\alpha>0$. If $\g\in\mathfrak{B}^\alpha$ and $\gamma\geq 0$, then there exists a sequence $\{\g_n\} \subset \mathfrak{B}^\alpha$ of compactly supported smooth finite-rank non-negative operators such that $\|\gamma_n\|_{\textup{op}}\leq\|\gamma\|_{\textup{op}}$, $\g_n\to\g$ in the strong operator topology, $\g_n\to \g$ in $\dot{\mathfrak{H}}^1$ and $\rho_{\g_n}\to\rho_\g$ in $L^1_{\text{loc}}(\R^d)$.
\end{lemma}

\begin{proof}
Given a non-negative operator $\gamma\in\mathfrak{B}^\alpha$, we define a sequence of non-negative operators $\{\gamma_n\}$ by $\g_n:=P_n \g P_n$, where $P_n:=\mathbf{1}_{\frac{1}{n}\leq(-\Delta)\leq n}$ is the Fourier multiplier operator with the symbol $\mathbf{1}_{\frac{1}{n} \leq |\xi|^2\leq n}$ ($\Rightarrow\|\gamma_n\|_{\textup{op}}\leq\|\gamma\|_{\textup{op}}$). In the following, we will prove that $\{\g_n\}$ is a sequence of trace-class operators with the desired convergences. If it is proved, approximating each trace-class operator $\gamma_n=\sum|\phi_{j,n}\rangle\langle\phi_{j,n}|$ by a finite-rank operator $\sum_{j=1}^{J_n}|\phi_{j,n}\rangle\langle\phi_{j,n}|$ with $J_n<\infty$, and then approximating $\{\phi_{j,n}\}_{j=1}^{J_n}$ by compactly supported smooth functions keeping orthogonality, we deduce the lemma.

Obviously, $\g_n$'s are trace-class operators, because
$$\|\g_n\|_{\S^1}\leq \|P_n|\nabla|^{-\alpha} \|_{\textup{op}}\| |\nabla|^\alpha\g|\nabla|^\alpha\|_{\S^1}\||\nabla|^{-\alpha} P_n\|_{\textup{op}}<\infty.$$
They converge to $\gamma$ in the strong operator topology, since for any $\phi\in L^2$,
\begin{align*}
\N{\g\phi-P_n\g P_n\phi}_{L^2} &\leq\N{(1-P_n)\g\phi}_{L^2}+\N{P_n\g (1-P_n)\phi}_{L^2}\\
 &\leq\N{(1-P_n)\g\phi}_{L^2}+\|P_n\|_{\textup{op}}\|\g\|_{\textup{op}}\|(1-P_n)\phi\|_{L^2}\to 0.
\end{align*}
Moreover, convergence in $\dot{\mathfrak{H}}^1$ follows from the dominated convergence theorem \cite[Theorem 2.16]{Simon}, because $||\nabla|P_n\g P_n|\nabla||\leq |\nabla|\g|\nabla|$ and $|\nabla|P_n\g P_n|\nabla|\rightharpoonup^* |\nabla|\g|\nabla|$ in $\mathfrak{S}^1$. Combining convergences in these two topologies, we get $\g_n\rightharpoonup^*\g$ (equivalently, $\gamma-\g_n(\geq 0)\rightharpoonup^*0)$ in $\mathfrak{B}^\alpha$. Therefore, it follows from Lemma \ref{lem:densitylocL1} that $\rho_{\g_n}\to\rho_\g$ in $L^1_{\text{loc}}$.
\end{proof}

\subsection{Kinetic energy inequality}\label{sec: kinetic energy inequality}

We now give a proof of the following kinetic energy inequality.

\begin{theorem}[Kinetic energy inequality]\label{theorem: general LT inequality}
Suppose $ \alpha>0$,
\begin{align*}
\left\{\begin{aligned}
\tfrac{d+2\alpha}{d}&\leq q<\infty &&\textup{when }\alpha\geq\tfrac{d}{2},\\
\tfrac{d+2\alpha}{d}&\leq q\leq\tfrac{d}{d-2\alpha} &&\textup{when }0<\alpha<\tfrac{d}{2}
\end{aligned}\right.
\end{align*}
and $\frac{1}{2q}=\frac{1}{2}-\frac{\theta\alpha}{d}$. Then, we have
\begin{equation}\label{eq: general LT}
\N{\rho_\gamma}_{L^q(\mathbb{R}^d)}\lesssim\|\gamma\|_{\textup{op}}^{1-\theta}\|\gamma\|_{\dot{\mathfrak{H}}^\alpha}^{\theta},\quad\forall \g\in\mathfrak{B}^\alpha.
\end{equation}
\end{theorem}

\begin{proof}
We claim that it is enough to show \eqref{eq: general LT} for compactly supported smooth finite-rank non-negative operators. To see this, we assume that \eqref{eq: general LT} is known for such operators, and let $\gamma\in\mathfrak{B}^\alpha$. If $\gamma\geq 0$, we take a sequence $\{\gamma_n\}$ from Lemma \ref{density}. Then, for any smooth cut-off $\chi\in C_c^\infty(\mathbb{R}^d)$, we have
\begin{align*}
\N{\chi\rho_\gamma}_{L^q}&\leq \N{\rho_{\gamma_n}}_{L^q}+\N{\chi(\rho_\gamma-\rho_{\gamma_n})}_{L^q}\\
&\lesssim\|\gamma_n\|_{\textup{op}}^{1-\theta}\|\gamma_n\|_{\dot{\mathfrak{H}}^\alpha}^{\theta}+o_n(1)\leq \|\gamma\|_{\textup{op}}^{1-\theta}\|\gamma\|_{\dot{\mathfrak{H}}^\alpha}^{\theta}+o_n(1),
\end{align*}
where $o_n(1)\to 0$ as $n\to\infty$. Since $\chi$ is arbitrary, it proves \eqref{eq: general LT} for non-negative $\gamma\in\mathfrak{B}^\alpha$. If $\g\in\mathfrak{B}^\alpha$ is not non-negative, we decompose $\gamma=\gamma_+-\gamma_-$ with $\gamma_\pm\geq 0$. Then, using the kinetic energy inequality for non-negative operators and the trivial inequalities $\N{\g_\pm}_{\textup{op}}\leq\|\gamma\|_{\textup{op}}$ and $\N{\g_\pm}_{\dot{\mathfrak{H}}^\alpha}\leq\|\gamma\|_{\dot{\mathfrak{H}}^\alpha}$, we prove the kinetic energy inequality,
\begin{align*}\|\rho_\gamma\|_{L^q}&\leq\|\rho_{\g_+}\|_{L^q}+\|\rho_{\g_-}\|_{L^q}\\
&\lesssim \|\gamma_+\|_{\textup{op}}^{1-\theta}\|\gamma_+\|_{\dot{\mathfrak{H}}^\alpha}^{\theta}+\|\gamma_-\|_{\textup{op}}^{1-\theta}\|\gamma_-\|_{\dot{\mathfrak{H}}^\alpha}^{\theta}\\
&\lesssim \|\gamma\|_{\textup{op}}^{1-\theta}\|\g\|_{\dH^\alpha}^\theta.
\end{align*}

Suppose that $\gamma$ is a compactly supported smooth finite-rank non-negative operator. Replacing $\gamma$ by $\frac{\gamma}{\|\gamma\|_{\textup{op}}}$, we may assume that $\|\gamma\|_{\textup{op}}=1$. Then, using the layer-cake representation of $(-\Delta)^\alpha$, i.e., $(-\Delta)^\alpha =\int_0^\infty P_ede$,
where $P_e$ is a spectral projection $\mathbf{1}_{\{(-\Delta)^\alpha\geq e\}}$, we write
	\begin{equation}\label{layer-cake}
	\Tr (-\Delta)^\alpha\gamma=\int_0^\infty \Tr(P_e\gamma P_e)\:de=\int_0^\infty\left(\int_{\mathbb{R}^d} \rho_e(x)\:dx\right)de,
	\end{equation}
where $\rho_e=\rho_{P_e\gamma P_e}$ is the density of the operator $P_e\gamma P_e$.

We claim that $\rho_e(x)$ has a lower bound, 
	\begin{equation}\label{LT claim}
	\rho_e(x)\geq \left( \sqrt{\rho(x)}-c_1 e^{\frac{1}{2(q-1)}}\right)_+^2,
	\end{equation}
where $A_+=\max\{A,0\}$, $\beta=\frac{d}{2}-\frac{\alpha}{q-1}$, and $c_1=(c\|\gamma\|_{\dot{\mathfrak{H}}^\alpha})^{\frac{\beta}{2\alpha}}$ for some constant $c>0$. Indeed, for a ball $B=B_r(x)$ of radius $r$ centered at $x$ in $\mathbb{R}^d$, by the triangle inequality, we have
	\begin{align}
	\frac{1}{|B|}\int_B\rho_e(y)\:dy&=\frac{1}{|B|}\Tr(\mathbf{1}_BP_e\gamma P_e\mathbf{1}_B)=\frac{1}{|B|}\big\|\mathbf{1}_{B}P_e\gamma^{\frac{1}{2}}\big\|_{\mathfrak{S}^2}^2\nonumber\\
	&\geq\bigg\{\frac{1}{|B|^\frac{1}{2}}\big\|\mathbf{1}_{B}\gamma^{\frac{1}{2}}\big\|_{\mathfrak{S}^2}-\frac{1}{|B|^\frac{1}{2}}\|\mathbf{1}_{B}P_e^\perp\gamma^{\frac{1}{2}}\big\|_{\mathfrak{S}^2}\bigg\}_+^2,\label{triangle inequality}
	\end{align}
where $|B|$ denotes the volume of the ball $B$ and $P_e^\perp:=\mathbf{1}_{\{(-\Delta)^\alpha< e\}}$ is the frequency projection orthogonal to $P_e$.
For the first term in the bracket $\{\cdots\}_+^2$ in \eqref{triangle inequality}, we use 
	\[\frac{1}{|B|}\big\|\mathbf{1}_{B}\gamma^{\frac{1}{2}}\big\|_{\mathfrak{S}^2}^2=\frac{1}{|B|}\Tr(\mathbf{1}_B\gamma\mathbf{1}_B)=\frac{1}{|B|}\int_B\rho_\gamma(y)\:dy\underset{r\to 0}\longrightarrow\rho_\gamma(x).\]
For the second term, by trivial inequalities, we prove that 
	\begin{align*}
	\frac{1}{|B|}\|\mathbf{1}_{B}P_e^\perp\gamma^{\frac{1}{2}}\|_{\mathfrak{S}^2}^2&=\frac{1}{|B|}\Tr(\mathbf{1}_BP_e^\perp\gamma P_e^\perp\mathbf{1}_{B})\\
	&\leq\frac{1}{|B|}\|\mathbf{1}_{B}P_e^\perp|\nabla|^{-\beta}\|_{\mathfrak{S}^2}^2\||\nabla|^\beta\gamma|\nabla|^\beta\|_{\textup{op}}\\
	&=\||\xi|^{-\beta}\mathbf{1}_{\{|\xi|\leq e^{\frac{1}{2\alpha}}\}}\|_{L^2_\xi}^2\||\nabla|^\beta\gamma|\nabla|^\beta\|_{\textup{op}}\\
	&\sim_{d,\alpha,q} e^{\frac{1}{q-1}}\||\nabla|^\beta\gamma|\nabla|^\beta\|_{\textup{op}}.
	\end{align*}
Note that $\beta=\frac{d}{2}-\frac{\alpha}{q-1}$ is non-negative, since we assumed that $q\geq 1+\frac{2\alpha}{d}$. When $\beta=0$, we have
	$$\frac{1}{|B|}\|\mathbf{1}_{B}P_e^\perp\gamma^{\frac{1}{2}}\|_{\mathfrak{S}^2}^2\lesssim_{d,\alpha} e^{\frac{1}{q-1}}.$$	
When $\beta>0$, interpolating the above two inequalities and the trivial embedding $\mathfrak{S}^1\hookrightarrow \mathcal{L}$, we prove that 
	$$\frac{1}{|B|}\|\mathbf{1}_{B}P_e^\perp\gamma^{\frac{1}{2}}\|_{\mathfrak{S}^2}^2\lesssim_{d,\alpha} e^{\frac{1}{q-1}}\N{|\nabla|^\alpha\gamma|\nabla|^\alpha}_{\textup{op}}^{\frac{\beta}{\alpha}}\leq e^{\frac{1}{q-1}}\|\gamma\|_{\dot{\mathfrak{H}}^\alpha}^{\frac{\beta}{\alpha}}.$$
Thus, sending $r\to0$ in \eqref{triangle inequality}, we prove the claim \eqref{LT claim}.

Going back to \eqref{layer-cake}, inserting the pointwise estimate \eqref{LT claim} and by the choice of $\alpha$ and $\beta$ ($\Rightarrow\frac{\beta(q-1)}{\alpha}=q\theta-1$), we get
	\begin{align*}
	\Tr (-\Delta)^\alpha\gamma&\geq \int_0^\infty\left(\int_{\mathbb{R}^d}\left( \sqrt{\rho(x)}-c_1 e^{\frac{1}{2(q-1)}}\right)_+^2\:dx\right)de\\
	&=\int_{\mathbb{R}^d}\bigg[\int_0^{\big(\frac{\rho(x)}{c_1^2}\big)^{q-1}}\rho(x)-2c_1\sqrt{\rho(x)}e^{\frac{1}{2(q-1)}}+c_1^2e^{\frac{1}{q-1}}\:de\bigg]\:dx\\
	&\sim\frac{1}{c_1^{2(q-1)}} \int_{\mathbb{R}^d}\rho(x)^qdx\sim \|\gamma\|_{\dot{\mathfrak{H}}^\alpha}^{-(q\theta-1)} \int_{\mathbb{R}^d}\rho(x)^qdx.
	\end{align*}
Thus, the kinetic energy inequality \eqref{eq: general LT} is proved.
\end{proof}

\subsection{Mid-frequency approximation for density functions}\label{sec: Low frequency approximation}

We close this section showing the mid-frequency approximation for density functions. This lemma will be used in the proof (Step 2) of the profile decomposition (Proposition \ref{profile decomposition}).

\begin{lemma}\label{mid-frequency approximation}
Suppose $ \alpha>0$ and
\begin{align*}
\left\{\begin{aligned}
\tfrac{d+2\alpha}{d}&<q<\infty &&\textup{when }\alpha\geq\tfrac{d}{2},\\
\tfrac{d+2\alpha}{d}&<q<\tfrac{d}{d-2\alpha} &&\textup{when }0<\alpha<\tfrac{d}{2}.
\end{aligned}\right.
\end{align*}
For large $A\geq 1$, we denote by $P_{med}=P_{med;A}$ the mid-frequency projection, that is, the Fourier multiplier with the symbol $\chi(\frac{\xi}{A})-\chi(A\xi)$, where $\chi \in C_c^\infty(\mathbb{R}^d)$ is a smooth cut-off such that $\chi(x)=1$ on $\{|x|\leq1\}$ but $\chi(x)=0$ on $\{|x|\geq2\}$. Then, there exists $\delta>0$ such that if $\gamma\in\mathfrak{B}^\alpha$ and $A>0$ is large enough, then
$$\|\rho_{\gamma}-\rho_{P_{med}\gamma P_{med}}\|_{L^q(\mathbb{R}^d)}\lesssim A^{-\delta},$$
where the implicit constant depends only on $\|\gamma\|_{\mathfrak{B}^\alpha}$.
\end{lemma}

\begin{proof}
Repeating the reduction in the beginning of the proof of Theorem \ref{theorem: general LT inequality}, we may assume that $\gamma$ is a compactly supported smooth finite-rank non-negative operator. We denote $P_{\leq A}:=\chi(\frac{\sqrt{-\Delta}}{A})$ and $P_{>A}:=1-P_{\leq A}$. First, we claim that 
\begin{equation}\label{med frequency approx claim}
\|\rho_{\gamma}-\rho_{P_{\leq A}\gamma P_{\leq A}}\|_{L^q}\lesssim A^{-\delta}.
\end{equation}
To show the claim, by the eigenfunction expansion $\gamma=\sum|\phi_j\rangle\langle\phi_j|=\sum_{j=1}^J|\phi_j\rangle\langle\phi_j|$, where $J<\infty$ and $\{\phi_j\}_{j=1}^J$ is an $L^2$-orthogonal set of compactly supported smooth functions, we write
$$(\rho_{\gamma}-\rho_{P_{\leq A}\gamma P_{\leq A}})(x)=\sum|\phi_j(x)|^2-|P_{\leq A}\phi_j(x)|^2.$$
We observe that by Cauchy-Schwarz, 
\begin{align*}
|(\rho_{\gamma}-\rho_{P_{\leq A}\gamma P_{\leq A}})(x)|&\leq\sum |P_{>A}\phi_j(x)|\left(|\phi_j(x)|+|P_{\leq A}\phi_j(x)|\right)\\
&\leq2\sum |P_{>A}\phi_j(x)||\phi_j(x)|\\
&\leq2\left\{\sum |P_{>A}\phi_j(x)|^2\right\}^{1/2}\left\{\sum |\phi_j(x)|^2\right\}^{1/2}\\
&=2\sqrt{\rho_{P_{>A}\gamma P_{>A}}(x)}\sqrt{\rho_{\gamma}(x)}.
\end{align*}
Hence, by the kinetic energy inequality \eqref{eq: general LT}, we get the bound, 
$$\|\rho_{\gamma}-\rho_{P_{\leq A}\gamma P_{\leq A}}\|_{L^q}\leq 2\|\rho_{P_{>A}\gamma P_{>A}}\|_{L^q}^{1/2}\|\rho_{\gamma}\|_{L^q}^{1/2}\lesssim \|\rho_{P_{>A}\gamma P_{>A}}\|_{L^q}^{1/2}\|\gamma\|_{\textup{op}}^{\frac{1-\theta}{2}}\|\gamma\|_{\dot{\mathfrak{H}}^\alpha}^{\frac{\theta}{2}}.$$
Thus, for the claim, it suffices to show that $\|\rho_{P_{>A}\gamma P_{>A}}\|_{L^q}\lesssim A^{-2\delta}$. Indeed, by the assumption $\frac{1}{q}>\max\{\frac{d-2\alpha}{d},0\}$, there exists $\epsilon>0$ such that $\frac{1}{2q}=\frac{1}{2}-\frac{\theta_\epsilon(\alpha-\epsilon)}{d}$ and $\theta_\epsilon:=\frac{d}{2q'(\alpha-\epsilon)}<1$. Then, by the kinetic energy inequality \eqref{eq: general LT} again but with $\theta_\epsilon>\theta$, we get
\begin{align*}
\N{\rho_{P_{>A}\gamma P_{>A}}}_{L^q}&\lesssim\|P_{>A}\gamma P_{>A}\|_{\textup{op}}^{1-\theta_\epsilon}\|P_{>A}\gamma P_{>A}\|_{\dot{\mathfrak{H}}^{\alpha-\epsilon}}^{\theta_\epsilon}\\
&\leq\|\gamma\|_{\textup{op}}^{1-\theta_\epsilon}A^{-2\epsilon\theta_\epsilon}\|\gamma\|_{\dot{\mathfrak{H}}^{\alpha}}^{\theta_\epsilon}.
\end{align*}
Therefore, taking $\delta=\epsilon\theta_\epsilon$, we prove the claim.

Repeating the above argument, we write
\begin{align*}
\|\rho_{P_{\leq A}\gamma P_{\leq A}}-\rho_{P_{med}\gamma P_{med}}\|_{L^q}&\leq 2\|\rho_{P_{\leq1/A}\gamma P_{\leq1/A}}\|_{L^q}^{1/2}\|\rho_{P_{\leq A}\gamma P_{\leq A}}\|_{L^q}^{1/2}\\
&\lesssim \|\rho_{P_{\leq1/A}\gamma P_{\leq1/A}}\|_{L^q}^{1/2}\|\gamma\|_{\textup{op}}^{\frac{1-\theta}{2}}\|\gamma \|_{\dot{\mathfrak{H}}^\alpha}^{\frac{\theta}{2}}.
\end{align*}
For the former factor, we take $\tilde{\theta}_\epsilon:=\frac{d}{2q'(\alpha+\epsilon)}$, which is non-negative for small $\epsilon>0$ by the assumption $q>\frac{d+2\alpha}{d}$, and then apply the kinetic energy inequality with $\tilde{\theta}_\epsilon$, 
\begin{align*}
\|\rho_{P_{\leq A}\gamma P_{\leq A}}-\rho_{P_{med}\gamma P_{med}}\|_{L^q}&\lesssim\|P_{\leq 1/A}\gamma P_{\leq 1/A}\|_{\textup{op}}^{\frac{1-\tilde{\theta}_\epsilon}{2}}\|P_{\leq 1/A}\gamma P_{\leq 1/A}\|_{\dot{\mathfrak{H}}^{\alpha+\epsilon}}^{\frac{\tilde{\theta}_\epsilon}{2}}\|\gamma\|_{\textup{op}}^{\frac{1-\theta}{2}}\|\gamma \|_{\dot{\mathfrak{H}}^\alpha}^{\frac{\theta}{2}}\\
&\lesssim A^{-\epsilon\tilde{\theta}_\epsilon} \|\gamma\|_{\textup{op}}^{\frac{2-\theta-\tilde{\theta}_\epsilon}{2}}\|\gamma \|_{\dot{\mathfrak{H}}^\alpha}^{\frac{\theta+\tilde{\theta}_\epsilon}{2}}\lesssim A^{-\delta}.
\end{align*}
Thus, combining with \eqref{med frequency approx claim}, we complete the proof of the lemma.
\end{proof}

\vspace{10pt}
\section{Profile decomposition associated with the kinetic energy inequality}
We establish the profile decomposition for bounded sequences of non-negative operators in the operator space $\mathfrak{B}^1$. {The profile decomposition is to capture the lack of compactness of embedding. Indeed, for \eqref{eq:LT} the noncompactness of spatial translation symmetry is responsible. The profile decomposition for many inequalities of functions are intensively studied and so now rather regarded as a common knowledge. Especially, if the embedding is inhomogeneous, and so no scaling invariance is involved, then the proof is fairly standard. Here, we work on an operator valued inequality. There arise several technical issues to be cautious. They are due to noncommutativity of operators, density argument, and self-adjoint truncation of operators. Hence, we include the proof for completeness.}


\begin{theorem}[Profile decomposition]\label{profile decomposition}
Suppose that the condition \eqref{admissible q} holds, that is,
\begin{align*}
\left\{\begin{aligned}
\tfrac{d+2}{d}&< q<\infty &&\textup{when }d=1,2,\\
\tfrac{d+2}{d}&<q<\tfrac{d}{d-2} &&\textup{when }d\geq 3,
\end{aligned}\right.
\end{align*}
Let $\{\gamma_n\}_{n=1}^\infty$ be a bounded sequence of non-negative operators in $\mathfrak{B}^1$.
Then, there exist $J^* \in \{0,1,2,\cdots \}\cup \{\infty\}$ and a subsequence of $\{\gamma_n\}_{n=1}^\infty$ (but still denoted by $\{\gamma_n\}_{n=1}^\infty$), non-negative operators $\{w^j \}_{j=1}^{J^*} \subset \mathfrak{B}^1$,  remainder operators $\{ R_n^J : n \in \mathbb{N},\ 1\le j \le J^* \} \subset \mathfrak{B}^1$ and translation parameters $\{x_n^j \in \mathbb{R}^d : n \in \mathbb{N},\ 1\le j \le J^*\}$ such that
	\begin{equation}\label{eq: profile decomposition}
	\gamma_n=\sum_{j=1}^J (\tau_{x_n^j}) w^j (\tau_{x_n^j})^*+R_n^J \qquad \text{for each } J \le J^*\footnote{Conventionally, if $ J^*=\infty $, $ 1\le J\le J^* $ this means $ 1 \le J <\infty $. Or if $J^*=0$, the summation is empty.},
	\end{equation}
where $\tau_y$ denotes the spatial translation operator by $ y \in \mathbb R^d$, i.e., $(\tau_y f)(x)\colonequals f(x-y)$, and the profile decomposition \eqref{eq: profile decomposition} satisfies the following properties.
\begin{enumerate}
\item (Asymptotic orthogonality of profiles)
\begin{align}
|x_n^j-x_n^{j'}|&\underset{n\to\infty}\longrightarrow \infty \quad\text{for all } j\neq j',\label{i-1}\\ 
(\tau_{x_n^j})^* R_n^J (\tau_{x_n^j})&\underset{n\to\infty}{\rightharpoonup^*} 0 \quad\text{ in }\mathfrak{B}^1 \text{ for each } 1\leq j\leq J^*. \label{i-2}
\end{align}
\item (Trace) It is obvious that 
\begin{equation}\label{ii-1}
\textup{Tr}(|\nabla|\gamma_n|\nabla|)=\sum_{j=1}^J\textup{Tr}(|\nabla| w^j|\nabla|)+\textup{Tr}(|\nabla| R_n^J|\nabla|).
\end{equation}
The remainder operator $R_n^J$ is not necessarily non-negative (see Remark \ref{remark on the profile decomposition} (2)). However, its trace is asymptotically non-negative, i.e., for each $J$, we have
\begin{equation}\label{ii-2}
\limsup_{n\to\infty}\textup{Tr}(|\nabla| R_n^J|\nabla|)\geq0.
\end{equation}
\item (Operator norm) For each $J$, we have
\begin{align}
\limsup_{n\to\infty}\|\gamma_n\|_{\textup{op}}&\geq \|w^j\|_{\textup{op}},\quad\text{for each }1\leq j\leq J,\label{iii-1}\\
\limsup_{n\to\infty}\|\gamma_n\|_{\textup{op}}&\geq\limsup_{n\to\infty}\|R_n^J\|_{\textup{op}}.\label{iii-2}
\end{align}
\item (Asymptotic orthogonality for the potential energy)
\begin{equation}\label{v}
\lim_{J\to J^*}\limsup_{n\to \infty}\bigg|\N{\rho_{\gamma_n}}_{L^q(\mathbb{R}^d)}^q-\sum_{j=1}^J\N{\rho_{w^j}}_{L^q(\mathbb{R}^d)}^q\bigg|=0.
\end{equation}
\item (Characterization of compactness)
\begin{equation}\label{vi}
\lim_{J\to J^*}\limsup_{n\to\infty}\|\rho_{R_n^J}\|_{L^q(\mathbb{R}^d)}=0.
\end{equation}
\end{enumerate}
\end{theorem}

\begin{remark}\label{remark on the profile decomposition}
The profile decomposition in Theorem \ref{profile decomposition} has a similar structure to that for functions, see for instance \cite[Proposition 3.1]{HmKe05}. However, there are several crucial differences, since we deal with operators. The following simple observations suggest us what we can or cannot prove.
\begin{enumerate}{\roman{enumii}}
\item The asymptotic Pythagorean expansion of the form
$$\N{\g_n}_{\textup{op}}=\sum_{j=1}^J\|\w^j\|_{\textup{op}}+\|R_n^J\|_{\textup{op}}+o_n(1)$$
does not hold, but we only have \eqref{iii-1} and \eqref{iii-2} instead. To see this, we consider the bounded sequence $\g_n=\sum_{j=1}^J |\phi(\cdot-jn)\rangle\langle\phi(\cdot-jn)|$, where $J<\infty$, $\phi\in C_c^\infty(\R^d)$ and $\textup{supp}\,\phi\subset B(0,1)$.
In this sequence, each profile $|\phi(\cdot-jn)\rangle\langle\phi(\cdot-jn)|$ has operator norm one, while $\|\gamma_n\|_{\textup{op}}=1$. 
\item The remainder term $R^J_n$ in the profile decomposition is not necessarily non-negative.
Indeed, the sequence $\gamma_n=|\phi+\phi(\cdot-n)\rangle\langle\phi+\phi(\cdot-n)|$, with $\phi\in C_c^\infty$, can be decomposed as
	\begin{align*}
	\gamma_n&=|\phi\rangle\langle\phi|+\Big\{(|\phi(\cdot-n)\rangle\langle\phi(\cdot-n)|+|\phi\rangle\langle\phi(\cdot-n)|+|\phi(\cdot-n)\rangle\langle\phi|)\Big\}\\
	&=:w^1+R_n^1,
	\end{align*}
and
	\begin{align*}
	\gamma_n&=|\phi\rangle\langle\phi|+(\tau_n)|\phi\rangle\langle\phi|(\tau_n)^*+\Big\{(|\phi\rangle\langle\phi(\cdot-n)|+|\phi(\cdot-n)\rangle\langle\phi|)\Big\}\\
	&=:w^1+(\tau_n)w^2(\tau_n)^*+R_n^2,
	\end{align*}
In this case, the remainders $|\nabla| R_n^1|\nabla|$ and $|\nabla| R_n^2|\nabla|$ have negative eigenvalues.
However, their traces are asymptotically non-negative: $\textup{Tr}(|\nabla| R_n^1|\nabla|)=\||\nabla|\phi\|_{L^2}^2+o_n(1)$ and $\textup{Tr}(|\nabla| R_n^2|\nabla|)=o_n(1)$.
\end{enumerate}
\end{remark}

\begin{remark}\label{remark 4}
In Theorem \ref{profile decomposition}, we assume that the sequence $\gamma_n$ is non-negative just for simplicity. Indeed, one can easily deduce the profile decomposition without non-negativity as follows. First, we decompose $\gamma_n=\gamma_{n,+}-\gamma_{n,-}$ with $\gamma_{n,\pm}\geq0$. Then, we apply Theorem \ref{profile decomposition} to each sequence $\{\gamma_{n,\pm}\}_{n=1}^\infty$ to obtain the profile decompositions, and sum them. Finally, if there are a profile $|\phi_+^j(\cdot-x_{n,+}^j)\rangle\langle \phi_+^j(\cdot-x_{n,+}^j)|$ from $\{\gamma_{n,+}\}_{n=1}^\infty$ and a profile $|\phi_-^{j'}(\cdot-x_{n,-}^{j'})\rangle\langle \phi_-^{j'}(\cdot-x_{n,-}^{j'})|$ from $\{\gamma_{n,-}\}_{n=1}^\infty$ but their translation parameters are not mutually asymptotically orthogonal, we combine two profiles as one in the sum.
\end{remark}

\begin{proof}[Proof of Theorem \ref{profile decomposition}]
Suppose that $\{\gamma_n\}_{n=1}^\infty$ is a bounded sequence of non-negative operators in $\mathfrak{B}^1$. Let $\mathcal{V}(\{\gamma_n\}_{n=1}^\infty)$ be the collection of weak-$\ast$ limits in $\mathfrak{B}^1$ of all possible subsequences of $\left\{(\tau_{x_n})^\ast\g_n(\tau_{x_n})\right\}_{n=1}^\infty$ for all possible sequence of translation parameters $\{x_n\}_{n=1}^\infty\subset\mathbb{R}^d$.
Due to Banach-Alaoglu theorem, $\mathcal{V}(\{\gamma_n\}_{n=1}^\infty)$ is a nonempty subset of $\mathcal B^1$. We denote
\begin{equation}\label{definition of eta}
\eta(\{\gamma_n\}_{n=1}^\infty):=\sup\Big\{\|w\|_{\dot{\mathfrak{H}}^1}:w\in\mathcal{V}(\{\gamma_n\}_{n=1}^\infty)\Big\}.
\end{equation}
We remark that contrary to the profile decomposition for bounded functions \cite[Proposition 3.1]{HmKe05}, we use the homogeneous norm $\|\cdot\|_{\dot{\mathfrak{H}}^1}$ in the definition of $\eta(\cdot)$ instead of the inhomogeneous norm $\|\cdot\|_{\mathfrak{H}^1}$. This is because the operator norm of the profile $w^j$ is not expected to converge to zero as $j\to \infty$ (see Remark \ref{remark on the profile decomposition} $(1)$).

\vspace{0.5em}
\noindent \textbf{Step 1. Construction of a profile decomposition} $ $ \\
 We construct a profile decomposition \eqref{eq: profile decomposition} for $\{\gamma_n\}_{n=1}^\infty$ by induction with the properties $(1)-(3)$ and 
\begin{equation}\label{eta}
\|w^J\|_{\dot{\mathfrak{H}}^1}\geq\frac{1}{2}\eta(\{R_n^{J-1}\}_{n=1}^\infty),\quad 1\le J\le J^*,
\end{equation}
where $R_n^0=\gamma_n$. The proof of the properties $(4)$ and $(5)$ is postponed in Step 2 and 3.

\vspace{0.5em}
\noindent\underline{The first induction step}

If $ \eta(\{\gamma_n\}_{n=1}^\infty)  =0$, in other words, all possible weak-$\ast$ limit of translated operators are zero, then we do not decompose, but just set $J^*=0 $. Otherwise, we can choose $w^1\in\mathcal{V}(\{\gamma_n\}_{n=1}^\infty)$ and a sequence $\{x_n^1\}_{n=1}^\infty \subset \mathbb{R}^d$ of translation parameters such that
$$\|w^1\|_{\dot{\mathfrak{H}}^1}\geq\frac{1}{2}\eta(\{\gamma_n\}_{n=1}^\infty),\quad\textup{and}\quad(\tau_{x_n^1})^\ast\g_n(\tau_{x_n^1})\rightharpoonup^* w^1\textup{ in $\mathfrak{B}^1$}$$
up to a subsequence. The weak-$\ast$ limit $w^1$ is non-negative, because so are $(\tau_{x_n^1})^\ast\g_n(\tau_{x_n^1})$'s. Set the remainder as 
	\[R_n^1:=\g_n-(\tau_{x_n^1})w^1(\tau_{x_n^1})^*.\]
Then, \eqref{i-2} and \eqref{ii-1} follow directly from the definition.

For \eqref{ii-2}, by the cyclicity of the trace, the commutativity of $ |\nabla|\tau_a = \tau_a |\nabla| $, and non-negativity of $\gamma_n$ and $w^1$, we get
	\begin{align*}
	\textup{Tr}(|\nabla| R_n^1|\nabla| )&=\textup{Tr}\left(|\nabla|\gamma_n|\nabla|\right)-\textup{Tr}(|\nabla| (\tau_{x_n^1})w^1(\tau_{x_n^1})^*|\nabla| )\\
	&=\textup{Tr}\left(|\nabla| (\tau_{x_n^1})^*\gamma_n(\tau_{x_n^1})|\nabla| \right)-\textup{Tr}(|\nabla| w^1|\nabla| )\\
	&=\|(\tau_{x_n^1})^*\gamma_n(\tau_{x_n^1}) \|_{\dot{\mathfrak{H}}^1}-\|w^1\|_{\dot{\mathfrak{H}}^1}\\
	&\geq o_n(1),
	\end{align*}
where in the last step, we used that $(\tau_{x_n^1})^\ast\gamma_n(\tau_{x_n^1})\rightharpoonup^*w^1$ in $\dot{\mathfrak{H}}^1$. \eqref{iii-1} is a direct consequence of $(\tau_{x_n^1})^*\gamma_n(\tau_{x_n^1})\rightharpoonup^* w^1$ in $\mathcal{L}$.

For \eqref{iii-2}, by self-adjointness, we write
	\begin{equation}\label{proof of iii-2,1}
	\|R_n^1\|_{\textup{op}}=\sup_{\|f\|_{L^2}=1}|\langle f|R_n^1|f\rangle|=\sup_{\|f\|_{L^2}=1}\left|\langle f|\gamma_n|f\rangle-\langle f|(\tau_{x_n^1}) w^1 (\tau_{x_n^1})^*|f\rangle\right|.
	\end{equation}
Note that both $\gamma_n$ and $(\tau_{x_n^1}) w^1 (\tau_{x_n^1})^*$ are non-negative. Hence, using the elementary inequality $|a-b|\leq\max\{a,b\}$ for $a,b\geq0$, we prove \eqref{iii-2},
\begin{equation}\label{proof of iii-2,2}
\begin{aligned}
\limsup_{n\to\infty}\|R_n^1\|_{\textup{op}}&\leq \limsup_{n\to\infty}\sup_{\|f\|_{L^2}=1}\max\Big\{\left|\langle f|\gamma_n|f\rangle\big|,\big|\langle f|(\tau_{x_n^1}) w^1 (\tau_{x_n^1})^*|f\rangle\right|\Big\}\\
&\leq \limsup_{n\to\infty}\max\Big\{\|\gamma_n\|_{\textup{op}},\|(\tau_{x_n^1}) w^1 (\tau_{x_n^1})^*\|_{\textup{op}}\Big\}\\
&=\limsup_{n\to\infty}\max\Big\{\|\gamma_n\|_{\textup{op}},\|w^1\|_{\textup{op}}\Big\}\\
&=\limsup_{n\to\infty}\|\gamma_n\|_{\textup{op}}\quad\textup{(by \eqref{iii-1}}.
\end{aligned}
\end{equation}
\vspace{0.5em}
\noindent\underline{The $(J+1)$-th induction step, assuming the $J$-th step}

Suppose that there is a profile decomposition \eqref{eq: profile decomposition} for $J$ with the properties $(1)-(3)$. If $\eta(\{R_n^J\}_{n=1}^\infty)=0$, then we are done and set $ J^*=J$. Otherwise, passing to a subsequence, we choose $w^{J+1}\neq0$ and $\{x_n^{J+1}\}_{n=1}^\infty$ such that
$$\|w^{J+1}\|_{\dot{\mathfrak{H}}^1}\geq\frac{1}{2}\eta(\{R_n^J\}_{n=1}^\infty)\quad\textup{and}\quad(\tau_{x_n^{J+1}})^*R_n^{J}(\tau_{x_n^{J+1}}) \rightharpoonup^* w^{J+1}\textup{ in }\mathfrak{B}^1.$$
and define $R_n^{J+1}=R_n^J-(\tau_{x_n^{J+1}})w^{J+1}(\tau_{x_n^{J+1}})^*$ so that \eqref{i-2} holds.

For \eqref{i-1}, by the induction hypothesis, it suffices to show that $|x_n^k-x_n^{J+1}|\to \infty$ for all $1\leq k\leq J$. If not, passing to a subsequence, we may assume that $(x_n^k-x_n^{J+1})\to a$ for some $1\leq k\leq J$ and $a\in\mathbb{R}^d$. Then it follows that
\begin{align*}
w^{J+1}&=\underset{n\to\infty}{w\textup{-}^\ast\lim}\ (\tau_{x_n^{J+1}})^*R_n^{J}(\tau_{x_n^{J+1}})\quad\textup{(in $\mathcal{L}$)}\\
&=\underset{n\to\infty}{w\textup{-}^*\lim}\ (\tau_{x_n^{J+1}})^*\bigg\{R_n^k-\sum_{j=k+1}^J(\tau_{x_n^{j}})w^{j}(\tau_{x_n^{j}})^*\bigg\}(\tau_{x_n^{J+1}})\\
&=\underset{n\to\infty}{w\textup{-}^*\lim}\ (\tau_a)^*\bigg\{(\tau_{x_n^k})^*R_n^k(\tau_{x_n^k})-\sum_{j=k+1}^J(\tau_{x_n^{j}-x_n^k})w^{j}(\tau_{x_n^{j}-x_n^k})^*\bigg\}(\tau_a)=0,
\end{align*}
because by the induction hypothesis, $(\tau_{x_n^k})^*R_n^k(\tau_{x_n^k}) \rightharpoonup^* 0$ in $\mathfrak{B}^1$ and $|x_n^j-x_n^k|\to \infty$. However, it contradicts to that $\w^{j+1}\neq0$. Therefore, \eqref{i-1} is proved.

Moreover, $w^{J+1}$ is non-negative, because
\begin{align*}
w^{J+1}&=\underset{n\to\infty}{w\textup{-}^*\lim}\ (\tau_{x_n^{J+1}})^*R_n^{J}(\tau_{x_n^{J+1}})\quad\textup{(in $\L$)}\\
&=\underset{n\to\infty}{w\textup{-}^*\lim}\ (\tau_{x_n^{J+1}})^*\bigg\{\gamma_n-\sum_{j=1}^J(\tau_{x_n^j})w^j(\tau_{x_n^j})^*\bigg\}(\tau_{x_n^{J+1}})\\
&=\underset{n\to\infty}{w\textup{-}^*\lim}\ (\tau_{x_n^{J+1}})^*\gamma_n(\tau_{x_n^{J+1}})\quad\textup{(by \eqref{i-1})},
\end{align*}
is the weak-$\ast$ limit of non-negative operators.

For \eqref{iii-1}, by the fact that $(\tau_{x_n^{J+1}})^*R_n^J(\tau_{x_n^{J+1}})\rightharpoonup^* w^{J+1}$ in $\L$ and the induction hypothesis,
$$\|w^{J+1}\|_{\textup{op}}\leq\liminf_{n\to\infty}\|(\tau_{x_n^{J+1}})^*R_n^J(\tau_{x_n^{J+1}})\|_{\textup{op}}=\liminf_{n\to\infty}\|R_n^J\|_{\textup{op}}\leq\limsup_{n\to\infty}\|\gamma_n\|_{\textup{op}}.$$

For \eqref{iii-2}, by the induction hypothesis ($\Rightarrow\sum_{j=1}^{J+1}(\tau_{x_n^j})w^j (\tau_{x_n^j})^*$ is non-negative) and the argument in \eqref{proof of iii-2,1} and \eqref{proof of iii-2,2}, we write
\begin{align*}
\limsup_{n\to\infty}\|R_n^{J+1}\|_{\textup{op}}&=\limsup_{n\to\infty}\Big\|\gamma_n-\sum_{j=1}^{J+1}(\tau_{x_n^j})w^j (\tau_{x_n^j})^*\Big\|_{\textup{op}}\\
&\leq\limsup_{n\to\infty}\sup_{\|f\|_{L^2}=1}\max\Big\{\big|\langle f|\gamma_n|f\rangle\big|,\Big|\Big\langle f\Big|\sum_{j=1}^{J+1}(\tau_{x_n^j})w^j (\tau_{x_n^j})^*\Big|f\Big\rangle\Big|\Big\}\\
&\leq \limsup_{n\to\infty}\max\Big\{\|\gamma_n\|_{\textup{op}},\Big\|\sum_{j=1}^{J+1}(\tau_{x_n^j})w^j (\tau_{x_n^j})^*\Big\|_{\textup{op}}\Big\}.
\end{align*}
Note that by \eqref{i-1} and the induction hypothesis for \eqref{iii-1},
\begin{align*}
\limsup_{n\to\infty}\Big\|\sum_{j=1}^{J+1}(\tau_{x_n^j})w^j (\tau_{x_n^j})^*\Big\|_{\textup{op}}&\leq \limsup_{n\to\infty}\sup_{1\leq j\leq J+1}\|(\tau_{x_n^j})w^j (\tau_{x_n^j})^*\|_{\textup{op}}\\
&=\sup_{1\leq j\leq J+1}\|w^j\|_{\textup{op}} \leq\limsup_{n\to\infty}\|\gamma_n\|_{\textup{op}}.
\end{align*}
Thus, \eqref{iii-2} follows.

It remains to prove \eqref{ii-2}. By the linearity of trace,
\begin{equation}\label{proof of ii-2}
\textup{Tr}(|\nabla| R_n^{J+1}|\nabla| )=\|\gamma_n\|_{\dot{\mathfrak{H}}^1}-\sum_{j=1}^{J+1}\|w^j\|_{\dot{\mathfrak{H}}^1}.
\end{equation}
Since $|\nabla| \gamma_n|\nabla| $ and $|\nabla| w^j|\nabla| $ are non-negative self-adjoint operators in $\mathfrak{S}^1$, they have eigenfunction expansions,
\begin{equation}\label{eigenfunction expansions}
|\nabla| \gamma_n|\nabla| =\sum_{k=1}^\infty \lambda_{n,k}|\phi_{n,k}\rangle\langle\phi_{n,k}|,\quad |\nabla| w^j|\nabla| =\sum_{\ell=1}^\infty\mu_\ell^j |\psi_\ell^j\rangle\langle\psi_\ell^j|,
\end{equation}
where $\lambda_{n,k}, \mu_\ell^j\geq0$ and $\{\phi_{n,k}\}_{n=1}^\infty$ and $\{\psi_{\ell}^j\}_{\ell=1}^\infty$ are orthonormal sets in $L^2$. Let $L\gg1$ be a large number to be chosen later.  Since by the asymptotic orthogonality of parameters \eqref{i-1}, $\{\psi_{\ell}^j(\cdot-x_n^j):1\leq j\leq J+1, \ell\geq1\}$ is an asymptotically orthonormal set in $L^2$, we have
\begin{equation}\label{proof of ii-2'}
\|\gamma_n\|_{\dot{\mathfrak{H}}^1}=\sum_{k=1}^\infty\lambda_{n,k} \|\phi_{n,k}\|_{L^2}^2\geq\sum_{k=1}^\infty\sum_{j=1}^{J+1}\sum_{\ell=1}^L \lambda_{n,k}\big|\langle\phi_{n,k}|\psi_{\ell}^j(\cdot-x_n^j)\rangle\big|^2+o_n(1).
\end{equation}
By algebra, 
\begin{align*}
\sum_{k=1}^\infty\lambda_{n,k}\big|\langle\phi_{n,k}|\psi_{\ell}^j(\cdot-x_n^j)\rangle\big|^2&= \Big\langle\psi_{\ell}^j(\cdot-x_n^j)\Big|\sum_{k=1}^\infty\lambda_{n,k}|\phi_{n,k}\rangle\langle\phi_{n,k}|\Big|\psi_{\ell}^j(\cdot-x_n^j)\Big\rangle\\
&=\Big\langle\psi_{\ell}^j\Big||\nabla| (\tau_{x_n^j})^*\gamma_n(\tau_{x_n^j})|\nabla| \Big|\psi_{\ell}^j\Big\rangle\\
&=\textup{Tr}\Big(|\psi_{\ell}^j\rangle\langle\psi_{\ell}^j||\nabla| (\tau_{x_n^j})^*\gamma_n(\tau_{x_n^j})|\nabla| \Big).
\end{align*}
Note that by \eqref{i-1} and \eqref{i-2},
\begin{align*}
(\tau_{x_n^j})^*\gamma_n(\tau_{x_n^j})&=(\tau_{x_n^j})^*\bigg(\sum_{j'=1}^j (\tau_{x_n^{j'}})w^{j'}(\tau_{x_n^{j'}})^*+R_n^j\bigg)(\tau_{x_n^j})\\
&=\sum_{j'=1}^j (\tau_{x_n^{j'}-x_n^j})w^{j'}(\tau_{x_n^{j'}-x_n^j})^*+(\tau_{x_n^j})^*R_n^j(\tau_{x_n^j})\rightharpoonup^* w^j
\end{align*}
in $\dot{\mathfrak{H}}^1$. Thus, by compactness of the rank-one operator $|\psi_{\ell}^j\rangle\langle\psi_{\ell}^j|$, it follows that
$$\sum_{k=1}^\infty\lambda_{n,k}\big|\langle\phi_{n,k}|\psi_{\ell}^j(\cdot-x_n^j)\rangle\big|^2=\textup{Tr}\Big(|\psi_{\ell}^j\rangle\langle\psi_{\ell}^j||\nabla| w^j|\nabla| \Big)+o_n(1)=\langle\psi_{\ell}^j||\nabla| w^j|\nabla||\psi_{\ell}^j\rangle +o_n(1).$$
Going back to \eqref{proof of ii-2}, we prove that for arbitrarily small $\epsilon>0$, there exists $L\gg1$ such that the following holds,
\begin{align*}
\textup{Tr}(|\nabla| R_n^{J+1}|\nabla| )&\geq \sum_{k=1}^\infty\sum_{j=1}^{J+1}\sum_{\ell=1}^L \lambda_{n,k}\big|\langle\phi_{n,k}|\psi_{\ell}^j(\cdot-x_n^j)\rangle\big|^2+o_n(1)-\sum_{j=1}^{J+1}\|w^j\|_{\dot{\mathfrak{H}}^1}\quad\textup{(by \eqref{proof of ii-2'})}\\
&=\sum_{j=1}^{J+1}\sum_{\ell=1}^L \langle\psi_{\ell}^j||\nabla| w^j|\nabla||\psi_{\ell}^j\rangle+o_n(1)-\sum_{j=1}^{J+1}\|w^j\|_{\dot{\mathfrak{H}}^1}\\
&=-\epsilon+o_n(1)\quad\textup{(by \eqref{eigenfunction expansions})}.
\end{align*}
Since $\epsilon>0$ is arbitrary, this proves \eqref{ii-2}.

Summarizing, by induction, we conclude that the profile decomposition satisfies $(1)-(3)$ and \eqref{eta}.

\vspace{0.5em}
\noindent \textbf{Step 2. Proof of \eqref{vi}} $ $\\
Fix any small $\epsilon>0$. We will show that $\|\rho_{R_n^J}\|_{L^q}\lesssim\epsilon$ if $J$ and $n$ are sufficiently large. First, we note that it suffices to estimate the mid-frequency approximation $\rho_{P_{med}R_n^JP_{med}}$, where $P_{med}=P_{med;A}$ is the mid-frequency projection given in Lemma \ref{mid-frequency approximation}. Indeed, by Lemma \ref{mid-frequency approximation}, there exists $\delta>0$ such that 
$$\|\rho_{R_n^J}-\rho_{P_{med}R_n^JP_{med}}\|_{L^q}\lesssim A^{-\delta}.$$
Here, the implicit constant does not depend on $n$ and $J$, because by \eqref{ii-1}--\eqref{iii-2}, all $R_n^J$'s are bounded uniformly in $J$ and $n$,
\begin{equation}\label{R_n^J uniform bounds}
\begin{aligned}
\|R_n^J\|_{\textup{op}}&\leq\|\gamma_n\|_{\textup{op}}+o_n(1)\lesssim 1,\\
\|R_n^J\|_{\dot{\mathfrak{H}}^1}&\leq \|\gamma_n\|_{\dot{\mathfrak{H}}^1}+\sum_{j=1}^J\|w^j\|_{\dot{\mathfrak{H}}^1}+o_n(1)\leq 2 \|\gamma_n\|_{\dot{\mathfrak{H}}^1}+o_n(1)\lesssim 1.
\end{aligned}
\end{equation}
Therefore, we can choose large $A>0$ such that $\|\rho_{R_n^J}-\rho_{P_{med}R_n^JP_{med}}\|_{L^q}\leq\epsilon$. Furthermore, approximating $R_n^J$ by Lemma \ref{density} with acceptable $O(\epsilon)$-error but still denoting by $R_n^J$, we may assume that all $R_n^J$'s are compactly supported, smooth and has finite-rank.

For $\rho_{P_{med}R_n^JP_{med}}$, by a trivial inequality, we get
	\begin{equation}\label{eq:low-low}
	\|\rho_{P_{med}R_n^J P_{med}}\|_{L^q}^q\leq\|\rho_{P_{med}R_n^J P_{med}}\|_{L^\infty}^{\delta}\|\rho_{P_{med}R_n^J P_{med}}\|_{L^{q-\delta}}^{q-\delta}.
	\end{equation}
Since we assume that $q>\frac{d}{d+2}$, the latter factor $\|\rho_{P_{med}R_n^J P_{med}}\|_{L^{q-\delta}}^{q-\delta}$ can be shown to be bounded using kinetic energy inequality \eqref{eq:LT} with $\theta_\delta=\frac{d(q-\delta-1)}{2(q-\delta)}$,
$$\|\rho_{P_{med}R_n^J P_{med}}\|_{L^{q-\delta}}\lesssim\|P_{med}R_n^J P_{med}\|_{\textup{op}}^{1-\theta_\delta}\|P_{med}R_n^J P_{med}\|_{\dot{\mathfrak{H}^1}}^{\theta_\delta}\leq \|R_n^J\|_{\textup{op}}^{1-\theta_\delta}\|R_n^J\|_{\dot{\mathfrak{H}^1}}^{\theta_\delta}\lesssim 1.$$
Therefore, it suffices to show smallness of the former factor $\|\rho_{P_{med}R_n^J P_{med}}\|_{L^\infty}$ in \eqref{eq:low-low}.

To estimate $\|\rho_{P_{med}R_n^J P_{med}}\|_{L^\infty}$, we observe that if $\gamma$ is a finite-rank operator of the form $\gamma=\sum|\phi_k\rangle\langle\phi_k|=\sum_{k=1}^K|\phi_k\rangle\langle\phi_k|$, where $\{\phi_k\}_{k=1}^K$ is an orthogonal set in $L^2$, then the Fourier transform of $\rho_{P_{med}\g P_{med}}=\sum|P_{med}\phi_j|^2$ is supported in a ball of radius $4A$ centered at $0$. Moreover, we have
\begin{equation}\label{finite-rank operator identities}
\begin{aligned}
\rho_\g(\cdot-\alpha)&=\sum |\phi_j(\:\cdot-\alpha)|^2=\rho_{(\tau_\alpha)^*\g(\tau_\alpha)},\\
\int V(x)\rho_\g(x)\:dx&=\int V(x)\sum |\phi_j(x)|^2\:dx=\Tr(V\g).
\end{aligned}
\end{equation}
For a radially symmetric smooth cut-off $\tilde{\chi}=\tilde{\chi}_A\in C_c^\infty(\mathbb{R}^d)$ such that $\tilde{\chi}(\xi)=1$ on $\{|\xi|\leq 4A\}$, let $\tilde{P}=\tilde{P}_{A}$ be the high-frequency cut-off defined by the Fourier multiplier operator with the symbol $\tilde{\chi}(\xi)$. Then, by the above observations, we may add a redundant projection $ \tilde{P}$, and write 
$$\rho_{P_{med}R_n^J P_{med}}(x)=\tilde{P}(\rho_{P_{med}R_n^J P_{med}})(x)=\int_{\R^d}\tilde{\chi}^\vee(x-y)\rho_{P_{med}R_n^J P_{med}}(y)dy.$$
Hence, there exists a sequence $\{x_n\}_{n=1}^\infty$ of translation parameters such that 
\begin{align*}
\|\rho_{P_{med}R_n^J P_{med}}\|_{L^\infty_x}&\leq \left|\int_{\R^d}\tilde{\chi}^\vee(x_n-y)\rho_{P_{med}R_n^J P_{med}}(y)dy\right|+\epsilon\\
&=\left|\int_{\R^d}\tilde{\chi}^\vee(y)\rho_{P_{med}R_n^J P_{med}}(y+x_n)dy\right|+\epsilon.
\end{align*}
By \eqref{finite-rank operator identities} and cyclicity of the trace, 
\begin{align*}
\int_{\R^d}\tilde{\chi}^\vee(y)\rho_{P_{med}R_n^J P_{med}}(y+x_n)dy&=\Tr\left(\tilde{\chi}^\vee P_{med}(\tau_{x_n})^*R_n^J (\tau_{x_n})P_{med}\right)\\
&=\Tr\left(P_{med}\tilde{\chi}^\vee P_{med}(\tau_{x_n})^*R_n^J (\tau_{x_n})\right).
\end{align*}
Moreover, by the Kato-Seiler-Simon inequality \eqref{Holder Schatten} as in the proof of Lemma \ref{lem:densitylocL1}, we can show that $|\nabla|^{-1}P_{med}\tilde{\chi}^\vee P_{med}|\nabla|^{-1}$ is a compact operator. Hence, by the definition of $\eta(\cdot)$ (see \eqref{definition of eta}), it follows that 
	\begin{align*}
	\|\rho_{P_{med}R_n^J P_{med}}\|_{L^\infty_x}&\leq \sup_{w\in\mathcal{V}(R_n^J)}\Tr\left(P_{med}\tilde{\chi}^\vee P_{med}w\right)+O(\epsilon)\\
	&\leq \sup_{w\in\mathcal{V}(R_n^J)}\||\nabla|^{-1}P_{med}\tilde{\chi}^\vee P_{med}|\nabla|^{-1}\|_{\textup{op}}\|w\|_{\dot{\mathfrak{H}}^1}+O(\epsilon)\\
	&\lesssim \||\nabla|^{-1}P_{med}\tilde{\chi}^\vee P_{med}|\nabla|^{-1}\|_{\textup{op}}\eta(\{R_n^J\}_{n=1}^\infty)+O(\epsilon)\\
	&\to O(\epsilon)\quad\textup{as }J\to\infty,
	\end{align*}
where in the last step, we used that $(2)$ and \eqref{eta}. Therefore, inserting this bound to \eqref{eq:low-low}, we conclude that $\|\rho_{P_{med}R_n^J P_{med}}\|_{L^q}=O(\epsilon)$. \\
\\
\noindent \textbf{Step 3. Proof of \eqref{v}} $ $\\
Note that, since $\g_n$ is a bounded sequence in $\mathfrak{B}^1$, by kinetic energy inequality (theorem \ref{theorem: LT inequality}),  $\{\rho_{\g_n}\}_{n=1}^\infty$ is a bounded sequence in $L^q$.
First, consider the single profile case, i.e., $J^*=1$.
Then $\gamma_n= (\tau_{x_n^1}) w^1 (\tau_{x_n^j})^*+R_n^1$, and
	\[\rho_{\g_n}(x+x_n^1)=\rho_{w^1}(x)+\rho_{R_n^1}(x+x_n^1).\]
In this case, $\N{\rho_{\g_n}(\cdot+x_n^1)-\rho_{w^1}}_{L^q}=\N{\rho_{R_n^1}(\cdot+x_n^1)}_{L^q}\to0$ (passing through a subsequence if neccesary).
So $\rho_{\g_n}(\cdot+x_n^1)\to\rho_{w^1}$ a.e. (passing through a subsequence if neccesary).
By refined Fatou lemma (or Brezis-Lieb lemma), we have
	\[\N{\rho_{\g_n}}_{L^q}^q-\N{\rho_{w^1}}_{L^q}^q \to 0,\]
i.e., the claim \eqref{v} follows.

If $ J^* \ge 2 $, first we fix $J$. In the profile decompositon, from \eqref{i-1}, at most one profile can be stationary (or bounded). We can set $ x_n^j \to \infty $ for $ j=2,\cdots, J $. If $x_n^1 \to \infty $, the argument is simpler so we omit the detail. We assume $x_n^1 $ is bounded and may further assume  $x_n^1 \equiv 0 $. 

\begin{lemma}\label{Lq lem}
Let $ g \in L^q(\R^d) $. Suppose $\{x_n\}_{n=1}^\infty $ is a sequence so that $ x_n \to \infty $. Then $ g(\cdot+x_n) \to 0 $ a.e.
\end{lemma}
\begin{proof}
We show that for any subsequence $\{x_{n_j} \} $ of $\{x_n\}$, there exists a subsequence $\{x_{n_{j_k}} \} $ such that $ g(\cdot +  x_{n_{j_k}}) \to 0 $ a.e. Fix a subsequence $ \{x_{n_j} \}$. It is not difficult to check $ g(\cdot+ x_{n_{j}} ) \to 0$ in $L^q_{loc}$. Then, via a diagonal argument we can show that there exists a further subsequence $ \{x_{n_{j_k}} \}$ so that $ g(\cdot +  x_{n_{j_k}}) \to 0 $ a.e.
\end{proof}

Due to Lemma \ref{Lq lem} and $ \rho_{w^j} \in L^q $, we have $ \rho_{\g_n}(x)-\rho_{R_n^J}(x)= \sum_{j=2}^J \rho_{w^j} (x - x_n^j) + \rho_{w_1}(x)  \to \rho_{w_1}(x) $ a.e..
Then, using the refined Fatou lemma, we obtain
	\[ \N{\rho_{\g_n}}_{L^q}^q  - \bigg\|\sum_{j=2}^J \rho_{w^j} (\cdot - x_n^j)\bigg\|_{L^q}^q  + \| \rho_{w_1} - \rho_{R_n^J}\|_{L^q}^q \to 0.  \]
Using \eqref{i-1}, Lemma \ref{Lq lem}, and refined Fatou lemma again, we have 
\[ \bigg\|\sum_{j=2}^J \rho_{w^j}(\cdot - x_n^j)\bigg\|_{L^q}^q  \to  \sum_{j=2}^J  \N{\rho_{w^j} }_{L^q}^q. \] 
Hence, we get
\[ \lim_{n\to \infty} \bigg| \N{\rho_{\g_n}}_{L^q}^q - \sum_{j=2}^J  \N{\rho_{w^j} }_{L^q}^q - \|\rho_{w^1} + \rho_{R^J_n}\|_{L^q}^q   \bigg|  =0 \]
After we take the limit in $n$ we are ready to take limit in $J$. For given $\epsilon >0 $, we choose $J_0$ large enough such that $ \| \rho_{R^J_n}\|_{L^q} \le \epsilon$ for $ J \ge J_0$.
Since
	\[\|\rho_{w^1}+ \rho_{R^J_n}\|_{L^q} -\N{\rho_{w^1}}_{L^q} \le \|\rho_{R^J_n}\|_{L^q} \le \epsilon, \]
we have 
	\[\lim_{n\to \infty} \bigg| \N{\rho_{\g_n}}_{L^q}^q - \sum_{j=1}^J  \N{\rho_{w^j} }_{L^q}^q  \bigg| \le \epsilon \N{\rho_{w^1}}_{L^q}^{q-1}. \]
In conclusion, we obtain 
	\[ \limsup_{J\to J^*} \limsup_{n\to \infty}  \bigg| \N{\rho_{\g_n}}_{L^q}^q - \sum_{j=1}^J \N{\rho_{w^j}}_{L^q}^q \bigg| =0.\]
\end{proof}

\vspace{10pt}

\section{Existence of an extremizer: Proof of Theorem \ref{main theorem} (1)}\label{sec:ExistGS}
We define the Weinstein functional by
$$W(\gamma):=\frac{\|\rho_\gamma\|_{L^q(\mathbb{R}^d)}^q}{\|\gamma\|_{\textup{op}}^{q(1-\theta)}\|\gamma\|_{\dot{\mathfrak{H}}^1}^{q\theta}},$$
where $\theta=\frac{d}{2q'}$. In this section, we prove existence of an extremizer for the kinetic energy inequality \eqref{eq:LT} from the maximization problem 
\begin{equation}\label{maximization problem}
W_{max}=\sup\left\{W(\gamma) : \gamma\in \mathfrak{B}^1\right\}.
\end{equation}

\begin{proposition}[Existence of a maximizer for \eqref{maximization problem}]\label{existence of a maximizer}
If \eqref{admissible q} holds, then there exists a non-negative operator $\mathcal{Q}\in \mathfrak{B}^1$ such that $W(\mathcal{Q})=W_{max}$, $\|\mathcal{Q}\|_{\textup{op}}=1$ and $\textup{Tr}\sqrt{-\Delta}\mathcal{Q}\sqrt{-\Delta}=\theta\|\rho_{\mathcal{Q}}\|_{L^q(\mathbb{R}^d)}^q$, where $\theta=\frac{d}{2q'}$.
\end{proposition}

  {The proof is based on the standard concentration-compactness argument, involving the profile decomposition for operators (Theorem \ref{profile decomposition}). Then we use an elementary inequality that asserts that a single profile is better than multiple profiles for maxmization. This type of inequality is called {\it a binding inequality}. See also \cite{LL,Lions1,Frank14}.} (or eliminating the dichotomy scenario in the concentration-compactness argument).

\begin{lemma}[Algebraic inequality]\label{algebraic inequality}
Suppose that $q>\frac{d+2}{d}$, $a_1, a_2, b_1,b_2>0$ and $\frac{a_1}{a_2^{q\theta}}\geq \frac{b_1}{b_2^{q\theta}}$, 
where $\theta=\frac{d}{2q'}$. Then, 
\begin{equation}\label{eq1 lem9}
\frac{a_1}{a_2^{q\theta}}\geq\frac{a_1+b_1}{(a_2+b_2)^{q\theta}}.
\end{equation}
If we further assume that $\delta\leq\frac{b_2}{a_2}\leq\frac{1}{\delta}$ for some $\delta>0$, there exists $c=c(\delta)>1$ such that 
\begin{equation}\label{eq2 lem9}
\frac{a_1}{a_2^{q\theta}}\geq c\cdot\frac{a_1+b_1}{(a_2+b_2)^{q\theta}}.
\end{equation}
\end{lemma}

\begin{proof}
By extracting the larger factor and by the assumption $(\Rightarrow\frac{b_1}{a_1}\leq(\frac{b_2}{a_2})^{q\theta})$, we write $\frac{a_1+b_1}{(a_2+b_2)^{q\theta}}=\frac{a_1}{a_2^{q\theta}}\frac{1+\frac{b_1}{a_1}}{(1+\frac{b_2}{a_2})^{q\theta}}\leq \frac{a_1}{a_2^{q\theta}}  f(\frac{b_2}{a_2})$, where $f(x)=\frac{1+x^{q\theta}}{(1+x)^{q\theta}}$ and $q\theta>1$ (from $q>\frac{d+2}{d}$). By elementary calculus, we see that $f(x)$ is a convex function on $[0,+\infty)$ such that $f(0)=1$ and $f(x)\to 1$ as $x\to\infty$. Hence, $f(x)\leq 1$, and $\max_{[\delta,\frac{1}{\delta}]}f(x)<1$ for $\delta >0$. Therefore, the inequalities \eqref{eq1 lem9} and \eqref{eq2 lem9} follow.
\end{proof}
\noindent We may assume that the maximizing sequence consists of non negative operators. 
\begin{lemma}[Positivity for maximization]\label{positivity of a maximizer}
If $\gamma=\gamma_+-\gamma_-$, $\gamma_\pm\in\mathfrak{B}^1$ and $\gamma_\pm\geq0$, then $W(\gamma)\geq\max_\pm W(\gamma_\pm)$.
\end{lemma}

\begin{proof}
We observe that $\|\gamma\|_{\textup{op}}=\max_\pm\|\gamma_\pm\|_{\textup{op}}$, $\|\gamma\|_{\dot{\mathfrak{H}}^1}=\sum_\pm\|\gamma_\pm\|_{\dot{\mathfrak{H}}^1}$ and
$$\|\rho_\gamma\|_{L^q}^q=\int_{\mathbb{R}^d}|\rho_{\gamma_+}-\rho_{\gamma_-}|^q dx\leq\int_{\mathbb{R}^d}\left(\max_\pm\rho_{\gamma_\pm}\right)^q dx\leq\sum_\pm\|\rho_{\gamma_\pm}\|_{L^q}^q.$$
Thus,
$$W(\gamma)\leq\frac{\sum_\pm\|\rho_{\gamma_\pm}\|_{L^q}^q}{\max_{\pm}\|\gamma_\pm\|_{\textup{op}}^{q(1-\theta)} (\|\gamma_+\|_{\dot{\mathfrak{H}}^1}+\|\gamma_-\|_{\dot{\mathfrak{H}}^1})^{q\theta}}.$$
Then, it follows from Lemma \ref{algebraic inequality} that 
$$W(\gamma)\leq\frac{1}{\max_{\pm}\|\gamma_\pm\|_{\textup{op}}^{q(1-\theta)}}\cdot\max_\pm\frac{\|\rho_{\gamma_\pm}\|_{L^q}^q}{\|\gamma_\pm\|_{\dot{\mathfrak{H}}^1}^{q\theta}}\leq\max_\pm\frac{\|\rho_{\gamma_\pm}\|_{L^q}^q}{\|\gamma_\pm\|_{\textup{op}}^{q(1-\theta)}\|\gamma_\pm\|_{\dot{\mathfrak{H}}^1}^{q\theta}}=\max_\pm W(\gamma_\pm).$$
\end{proof}

\begin{proof}[Proof of Proposition \ref{existence of a maximizer}] $ $ \\
Suppose that $\{\gamma_n\}_{n=1}^\infty\subset\mathfrak{B}^1$ is a maximizing sequence for the Weinstein functional $W(\gamma)$. By Lemma \ref{positivity of a maximizer}, we may assume that $\gamma_n$ is either non-negative or {non-positve}. When it is {non-positve}, replacing $\gamma_n$ by $-\gamma_n$, we may assume that $\gamma_n\geq 0$. For $\lambda>0$, we denote by $\sigma_\lambda$ the scaling transformation $\sigma_\lambda f(x)=f(\tfrac{x}{\lambda})$. Note that the Weinstein functional $W(\gamma)$ is invariant under constant multiplication $\gamma\mapsto a\gamma$ and scaling $\gamma\mapsto(\sigma_{\lambda})^*\gamma(\sigma_{\lambda})$. Hence, replacing $\gamma_n$ by $a_n(\sigma_{\lambda_n})^*\gamma_n(\sigma_{\lambda_n})$ with a suitable choice of $a_n, \lambda_n>0$, we may normalize so that $\|\gamma_n\|_{\textup{op}}=\|\gamma_n\|_{\dot{\mathfrak{H}}^1}=1$ $(\Rightarrow W(\gamma_n)=\|\rho_{\gamma_n}\|_{L^q}^q)$.

We apply the profile decomposition (Theorem \ref{profile decomposition}) to the bounded sequence $\{\gamma_n\}_{n=1}^\infty$, and write it as a sum of two pieces,
$$\gamma_n=(\tau_{x_n})w^1(\tau_{x_n})^*+R_n^1,\quad\textup{with}\quad w^1\neq 0,$$
passing to a subsequence. Indeed, if there is no such non-zero $w^1$, then $W(\gamma_n) =\N{\rho_{\gamma_n}}_{L^q}^q \to 0 $. It contradicts to the maximality of the sequence $\{\gamma_n\}_{n=1}^\infty$. Moreover, since the Weinstein functional $W(\gamma)$ is invariant under translation $\gamma\mapsto(\tau_{x})^*\gamma(\tau_{x})$, replacing $\gamma_n$ by $(\tau_{x_n})^*\gamma_n(\tau_{x_n})$, we may write
$$\gamma_n=w^1+R_n^1,\quad\textup{with}\quad w^1\neq 0,$$
where the translated remainder $(\tau_{x_n})^*R_n^1(\tau_{x_n})$ is still denoted by $R_n^1 $.

We will show that $w^1$ must be a maximizer. Indeed, by the asymptotic orthogonality of profiles, we may separate profiles in the functional up to negligible error, 
	\[W(\gamma_n)=\frac{\N{\rho_{w^1}}_{L^q}^q+\|\rho_{R_n^1}\|_{L^q}^q}{(\N{w^1}_{\dH^1}+\N{R_n^1}_{\dH^1})^{q\theta}}+o_n(1).\]
Obviously, we have either
	\[\frac{\|\rho_{w^1}\|_{L^q}^q}{\|w^1\|_{\dot{\mathfrak{H}}^1}^{q\theta}}\geq \frac{\|\rho_{R_n^1}\|_{L^q}^q}{\|R_n^1\|_{\dot{\mathfrak{H}}^1}^{q\theta}}\quad\text{or}\quad \frac{\|\rho_{R_n^1}\|_{L^q}^q}{\|R_n^1\|_{\dot{\mathfrak{H}}^1}^{q\theta}}\geq \frac{\|\rho_{w^1}\|_{L^q}^q}{\|w^1\|_{\dot{\mathfrak{H}}^1}^{q\theta}}.\]
In the former case, since $1=\|\gamma_n\|_{\textup{op}}\geq \|w^1\|_{\textup{op}}+o_n(1)$ by \eqref{iii-1}, it follows from the inequality \eqref{eq1 lem9} that $W_{max}\leftarrow W(\gamma_n)\leq W(w^1)+o_n(1)$, so $w^1$ is a maximizer. In the latter case, $\N{R_n^1}_{\dot{\mathfrak{H}}^1}\to 0$ up to a subsequence. Indeed, otherwise there exists $\epsilon>0$ such that $\N{R_n^1}_{\dot{\mathfrak{H}}^1}\geq \epsilon$ up to a subsequence, as well as $w^1 \ne 0$. Hence, it follows from \eqref{eq2 lem9} and \eqref{iii-2} that $cW(\gamma_n)\leq W(R_n^1)+o_n(1)$ for some $c>1$ independent of $n$, which contradicts to the maximality of $\{\gamma_n\}_{n=1}^\infty$. If $\N{R_n^1}_{\dot{\mathfrak{H}}^1}\to 0$, then $\|\rho_{R_n^1}\|_{L^q}\to 0$ by Theorem \ref{theorem: LT inequality}. Thus,
	\[W_{max}\leftarrow W(\gamma_n)=\frac{\N{\rho_{w^1}}_{L^q}^q}{\N{w^1}_{\dH^1}^{q\theta}}+o_n(1)\leq W(w^1)+o_n(1),\]
that is, $w^1$ is a maximizer.

Let $w\in\mathfrak{B}^1$ be a maximizer for the Weinstein functional. Then, by \eqref{positivity of a maximizer}, it must be non-negative. Moreover, replacing $w$ by $\mathcal{Q}=a(\sigma_{\lambda})^*w(\sigma_{\lambda})$ with suitable $a,\lambda>0$, we may make $\mathcal{Q}$ satisfy $\|\mathcal{Q}\|_{\textup{op}}=1$ and $\textup{Tr}\sqrt{-\Delta}\mathcal{Q}\sqrt{-\Delta}=\theta\|\rho_{\mathcal{Q}}\|_{L^q}^q$.
\end{proof}

\vspace{10pt}

\section{Derivation of the Euler-Lagrange equation:  \\Proof of Theorem \ref{main theorem} (2) and (3)}

We derive the Euler-Lagrange equation for the extremizer for the kinetic energy inequality constructed in the previous section. Throughout this section, we assume \eqref{admissible q}, and denote by $\mathcal{Q}$ the extremizer in Proposition \ref{existence of a maximizer}.

The first step is to transfer the maximization problem \eqref{maximization problem} to the following minimization problem
\begin{equation}\label{relative energy}
I_{V; min}=\inf\left\{I_V(\gamma): 0\leq\g\leq1\textup{ and }\g\in\mathfrak{B}^1 \right\},
\end{equation}
where
$$I_V(\gamma):=\Tr(|\nabla|\g|\nabla|)-\int_{\R^d}V(x)\rho_\g(x)dx.$$

\begin{lemma}\label{lem5.1}
$\mathcal{Q}$ is a minimizer for \eqref{relative energy} with $V=\rho_{\mathcal{Q}}^{q-1}$, i.e., $I_{V; min}=I_V(\mathcal{Q})$.
\end{lemma}
\begin{proof}
Fix $\gamma\in\mathfrak{B}^1$. Then, for $0\leq t\leq 1$, the operator $(1-t)\mathcal{Q}+t\g$ is still contained in $\mathfrak{B}^1$, so it is admissible for the maximization problem \eqref{maximization problem}. Inserting it into the Weinstein functional and differentiating at $t=0$,  we write
\begin{align*}
\frac{d}{dt} \bigg|_{t=0}W((1-t)\mathcal{Q}+t\g)&=\frac{qW(\mathcal{Q})}{\|\rho_{\mathcal{Q}}\|_{L^q}^q}\left\{-\Tr(|\nabla|(\g-\mathcal{Q})|\nabla|)+\int_{\R^d}\rho_{\mathcal{Q}}^{q-1}(\rho_\g-\rho_{\mathcal{Q}})dx\right\}\\
&\quad-\frac{q(1-\theta)W(\mathcal{Q})}{\|\mathcal{Q}\|_{\textup{op}}}\frac{d}{dt} \bigg|_{t=0} \|(1-t)\mathcal{Q}+t\g\|_{\textup{op}},
\end{align*}
where we used that $\|\mathcal{Q}\|_{\dot{\mathfrak{H}}^1}=\theta\|\rho_{\mathcal{Q}}\|_{L^q}^q$ in calculation. Note that $\frac{d}{dt}|_{t=0}W((1-t)\mathcal{Q}+t\g)\leq 0$ and $\frac{d}{dt}|_{t=0} \|(1-t)\mathcal{Q}+t\g\|_{\textup{op}} \leq 0$, because $\mathcal{Q}$ maximizes the Weinstein functional and $\|(1-t)\mathcal{Q}+t\g\|_{\textup{op}} \leq 1$. Therefore, it follows that
$$\int_{\R^d}\rho_{\mathcal{Q}}^{q-1}(\rho_\g-\rho_{\mathcal{Q}})dx-\Tr(|\nabla|(\g-\mathcal{Q})|\nabla|)\leq 0,$$
equivalently $I_V(\mathcal{Q}) \leq I_V(\gamma)$. Since $\gamma$ is arbitrary, we conclude that $\mathcal{Q}$ is a minimizer.
\end{proof}

Under the assumption \eqref{admissible q}, the Schr\"odinger operator $(-\Delta-\rho_Q^{q-1})$ is self-adjoint.
More precisely, for
	\[\left\{\begin{array}{ll}
	\frac{d+2}{d}< q <\infty,& \text{when }d=1, 2,\\
	\frac{5}{3}<q\leq\frac{5}{2} &\text{when }d=3,\\
	\frac{d+2}{d}< q < \frac{d}{d-2}, & \text{when }d\geq4,
	\end{array}\right.\]
$(-\Delta-\rho_{\mathcal{Q}}^{q-1})$ is essentially self-adjoint on $C_0^\infty(\R^d)$ (see \cite[Theorem X.29]{ReSi2}).
In particular, if $d=3$, then $\rho_{\mathcal{Q}}^{q-1}\in L^2(\R^3)+L^\infty(\R^3)$, and the result comes from the Kato-Rellich theorem \cite[Theorem X.15]{ReSi2}.
On the other hand, if $d=3$, and $\frac{19}{9}\leq q<3$, then $\rho_{\mathcal{Q}}^{q-1}\in L^\frac{3}{2}(\R^3)$, and by the KLMN theorem, one can view $(-\Delta-\rho_{\mathcal{Q}}^{q-1})$ as Friedrichs self-adjoint realization. 
Moreover, if $d\geq 3$, and $\big(\frac{d+2}{d}<\big)\frac{d^2+2d+4}{d^2}\leq q<\frac{d}{d-2}$, then $\rho_{\mathcal{Q}}^{q-1}\in L^\frac{d}{2}(\R^d)$, and the Cwikel-Lieb-Rozenblum bound \cite[Theorem XIII.12]{ReSi4} implies that $(-\Delta-\rho_{\mathcal{Q}}^{q-1})$ has only finitely many negative eigenvalues.

We denote by $\{-\mu_j\}_{j=1}^{J_-}$ the set of negative eigenvalues for the Schr\"odinger operator $(-\Delta-\rho_{\mathcal{Q}}^{q-1})$ with $J_-\in\mathbb{N}\cup\{\infty\}$ and the ordering $-\mu_1\leq -\mu_2\leq-\mu_3\leq\cdots <0$. Let $\phi_j^-$ be the $L^2$-normalized eigenfunction corresponding to the eigenvalue $-\mu_j$, that is, the $H^1$-weak solution to 
\begin{equation}\label{EL for negative e.f.}
(-\Delta-\rho_{\mathcal{Q}}^{q-1})\phi_j^-=-\mu_j\phi_j^-.
\end{equation}

Next, we characterize minimizers for the minimization problem \eqref{relative energy} with $V=\rho_{\mathcal{Q}}^{q-1}$. Note that if $\gamma$ is a smooth finite-rank operator, then
$$I_V(\gamma)=\textup{Tr}(-\Delta-\rho_{\mathcal{Q}}^{q-1})\gamma.$$
From the above expression, we see that to achieve the smallest value, $\gamma$ should contain $|\phi_j^-\rangle\langle \phi_j^-|$'s as many as possible to be more negative, but its spectrum should not include any positive spectrum for the Schr\"odinger operator $(-\Delta-\rho_{\mathcal{Q}}^{q-1})$. Moreover, adding a self-adjoint operator $\mathcal{Q}_0$ acting on the eigenspace associated with the zero eigenvalue does not change the value of the functional. Therefore, one may expect that the spectral projection to the negative discrete spectrum
\begin{equation}\label{projection to -}
\Pi^-=\sum_{j=1}^{J_-}|\phi_j^-\rangle\langle\phi_j^-|
\end{equation}
is a minimizer for \eqref{relative energy}, and so is $\Pi^-+\mathcal{Q}_0$. This is justified below.

\begin{lemma}\label{lem5.2}
If $\gamma_{min}$ is a minimizer for \eqref{relative energy} with $V=\rho_{\mathcal{Q}}^{q-1}$, then $\gamma_{min}=\Pi^-+\mathcal{Q}_0$, where $\Pi^-$ is given by \eqref{projection to -} and $\mathcal{Q}_0$ is a self-adjoint operator acting on the eigenspace associated with the zero eigenvalue for the Schr\"odinger operator $(-\Delta-\rho_{\mathcal{Q}}^{q-1})$. Here, $\mathcal{Q}_0$ could be zero.
\end{lemma}

\begin{proof}
We claim that $\Pi^-\in\mathfrak{B}^1$ (so it is admissible for the variational problem \eqref{relative energy}). Indeed, if $J_-<\infty$, the claim is obvious. We assume that $J_-=\infty$, and consider $\Pi_n^-=\sum_{j=1}^n |\phi_j^-\rangle\langle\phi_j^-|$. Since $\{\phi_j^-\}_{j=1}^\infty$ is an orthonormal set in $L^2$, $\Pi_n^-\to\Pi^-$ as $n\to\infty$ in the strong operator topology. We also observe that by \eqref{EL for negative e.f.},
$$I_V(\Pi_n^-)=\Tr(|\nabla|\Pi_n^-|\nabla|)-\int_{\R^d}V\rho_{\Pi_n^-}dx=\sum_{j=1}^n \|\nabla\phi_j^-\|_{L^2}^2-\langle V\phi_j^-, \phi_j^-\rangle_{L^2}=-\sum_{j=1}^n\mu_j <0.$$
Therefore, we have
\begin{align*}
\|\Pi_n^-\|_{\dot{\mathfrak{H}}^1}&=\Tr(|\nabla|\Pi_n^-|\nabla|)\\
&<\int_{\R^d}V\rho_{\Pi_n^-}dx+\left(\int_{\R^d}V\rho_{\Pi_n^-}dx-\Tr(|\nabla|\Pi_n^-|\nabla|)\right)=-I_{2V}(\Pi_n^-)\\
&\leq-I_{2V;min}<\infty,
\end{align*}
where in the last step, the bound for $-I_{2V;min}$ follows from Lemma \ref{lem5.1} with different scaling. Together with the strong convergence $\Pi_n^-\to\Pi^-$ in the operator norm, the uniform $\dot{\mathfrak{H}}^1$-boundedness implies that $\|\Pi^-\|_{\dot{\mathfrak{H}}^1}<\infty$. Therefore, we prove the claim.

Next, we show that $\Pi^-$ is a minimizer. Given a smooth finite-rank operator $\g\in\mathfrak{B}^1$ such that $0\leq\g\leq 1$, we write 
\begin{equation}\label{lemma 5.2 proof 1}
I_V(\gamma)=\Tr(|\nabla|\g|\nabla|-V\g)=I_V(\Pi^-)+\Tr\left(|\nabla|(\g-\Pi^-)|\nabla|-V(\g-\Pi^-)\right).
\end{equation}
We claim that if $\tilde{\gamma}\in\mathfrak{B}^1$ and $-\Pi^-\leq\tilde{\g}\leq 1-\Pi^-$, then
\begin{equation}\label{lemma 5.2 proof 2}
\Tr\left(|\nabla|\tilde{\g}|\nabla|-V\tilde{\g}\right)\geq 0.
\end{equation}
Indeed, by density (Lemma \ref{density}) and linearity, it suffices to show \eqref{lemma 5.2 proof 2} for one-particle projections of the form $-|\phi\rangle\langle\phi|$ or $|\psi\rangle\langle\psi|$. By the assumption $-\Pi^-\leq\tilde{\g}\leq 1-\Pi^-$, $\phi$ must be contained in $\textup{span}\{\phi_j^-\}_{j=1}^{J_-}$, while $\psi$ should be orthogonal to $\textup{span}\{\phi_j^-\}_{j=1}^{J_-}$. Therefore, it follows that
\begin{equation}\label{lemma 5.2 proof 3}
\Tr\left(|\nabla|(-|\phi\rangle\langle\phi|)|\nabla|-V(-|\phi\rangle\langle\phi|)\right)=-\left\{\|\nabla\phi\|_{L^2}^2-\int_{\mathbb{R}^d}V|\phi|^2 dx\right\}> 0
\end{equation}
while
\begin{equation}\label{lemma 5.2 proof 4}
\Tr\left(|\nabla|(|\psi\rangle\langle\psi|)|\nabla|-V(|\psi\rangle\langle\psi|)\right)=\|\nabla\psi\|_{L^2}^2-\int_{\mathbb{R}^d}V|\psi|^2 dx\geq 0.
\end{equation}
Thus, the claim \eqref{lemma 5.2 proof 2} is proved. Coming back to the identity \eqref{lemma 5.2 proof 1}, by the claim \eqref{lemma 5.2 proof 2}, we get $I_V(\gamma)\geq I_V(\Pi^-)$. Then, by density again, $I_{V;min}\geq I_V(\Pi^-)$ and so $\Pi^-$ is a minimizer.

Suppose that $\gamma_{min}$ is another minimizer. Then, $\mathcal{Q}_0:=\gamma_{min}-\Pi^-$ satisfies
$$I_{V;min}=I_V(\gamma_{min})=I_V(\Pi^-)+I_V(\mathcal{Q}_0)=I_{V;min}+I_V(\mathcal{Q}_0)$$
and consequently,
$$I_V(\mathcal{Q}_0)=\Tr(|\nabla|\mathcal{Q}_0|\nabla|)-\int_{\R^d}V\rho_{\mathcal{Q}_0}dx=0.$$
Therefore, we conclude that $\mathcal{Q}_0$ must act on the eigenspace associated with the zero eigenvalue, because otherwise by \eqref{lemma 5.2 proof 3} and \eqref{lemma 5.2 proof 4}, $I_V(\mathcal{Q}_0)>0$ (contradiction!).
\end{proof}

\begin{proof}[Proof of Theorem \ref{main theorem} $(2)$ and $(3)$] $ $\\
Since $L^2(\mathbb{R}^d)$ is separable, the operator $\mathcal{Q}_0$ in Lemma \ref{lem5.2} (if it is non-zero) can be written as $\mathcal{Q}_0=\sum_{k=1}^{K_0}|\phi_k^0\rangle\langle\phi_k^0|$ for some $L^2$-orthogonal set $\{\phi_k^0\}_{k=1}^{K_0}\subset\textup{Ker}(-\Delta-\rho_{\mathcal{Q}}^{q-1})$. Thus, combining Lemma \ref{lem5.1} and \ref{lem5.2}, we prove Theorem \ref{main theorem} $(2)$. The Euler-Lagrange equations \eqref{EL} and \eqref{EL} are derived from \eqref{EL for negative e.f.} and $\phi_k^0\in\textup{Ker}(-\Delta-\rho_{\mathcal{Q}}^{q-1})$.
\end{proof}

\vspace{10pt}

\section{Properties of an extremizer: Proof of Theorem \ref{main theorem} (4)-(6)}

We again fix the extremizer $\mathcal{Q}$ in Proposition \ref{existence of a maximizer} that satisfies the properties in Theorem Theorem \ref{main theorem} (1)-(3). In this section, we show additional properties in Theorem \ref{main theorem} (4)-(6). Throughout this section, for notational convenience, combining two sums in \eqref{spectral decomposition for Q}, we write
$$\mathcal{Q}=\sum_{j=1}^J|\phi_j\rangle\langle\phi_j|$$
for some $J\in\mathbb{N}\cup\{\infty\}$, where $\{\phi_j\}_{j=1}^J$ is an orthogonal set in $L^2$, and each $\phi_j$ is a $H^1$-weak solution to
\begin{equation}\label{EL for all}
(-\Delta-\rho_{\mathcal{Q}}^{q-1})\phi_j=-\mu_j\phi_j\quad\textup{with }\mu_j\geq 0.
\end{equation}

By construction, we only have a bound on $\textup{Tr}\sqrt{-\Delta}\mathcal{Q}\sqrt{-\Delta}$. However, using \eqref{EL for all}, we can upgrade this trivial regularity.

\begin{proof}[Proof of Theorem \ref{main theorem} $(4)$]
For $\epsilon, R>0$, let $P=P_{\epsilon, R}$ be a frequencies truncation defined by $\widehat{Pf}(\xi)=\mathbf{1}_{\epsilon\leq|\xi|\leq R}\hat{f}(\xi)$. We observe that for any $\alpha\in\mathbb{R}$,
\begin{equation}\label{truncation trick}
\begin{aligned}
\|(-\Delta+\mu_j)P\phi_j\|_{\dot{H}^\alpha}^2&=\langle\phi_j, (-\Delta+\mu_j)(-\Delta+\mu_j)(-\Delta)^\alpha P^2\phi_j\rangle_{L^2}\\
&=\langle \rho_{\mathcal{Q}}^{q-1}\phi_j, (-\Delta+\mu_j)(-\Delta)^\alpha P^2\phi_j\rangle_{L^2}\\
&=\langle (-\Delta+\mu_j)(-\Delta)^{\alpha}P^2(\rho_{\mathcal{Q}}^{q-1}\phi_j), \phi_j\rangle_{L^2}\\
&=\langle (-\Delta)^\alpha P^2(\rho_{\mathcal{Q}}^{q-1}\phi_j), \rho_{\mathcal{Q}}^{q-1}\phi_j\rangle_{L^2}\\
&=\|P(\rho_{\mathcal{Q}}^{q-1}\phi_j)\|_{\dot{H}^\alpha}^2,
\end{aligned}
\end{equation}
since $\phi_j$ is a $H^1$-weak solution to \eqref{EL for all}. By the observation, if $d=1,2$, then
\begin{align*}
\sum \|P\phi_j\|_{\dot{H}^2}^2&\leq\sum \|(-\Delta+\mu_j)P\phi_j\|_{L^2}^2=\sum \|P(\rho_{\mathcal{Q}}^{q-1}\phi_j)\|_{L^2}^2\\
&\leq\sum \|\rho_{\mathcal{Q}}^{q-1}\phi_j\|_{L^2}^2=\sum \int_{\mathbb{R}^d}\rho_{\mathcal{Q}}^{2(q-1)}|\phi_j|^2 dx=\|\rho_{\mathcal{Q}}\|_{L^{2q-1}}^{2q-1}<\infty,
\end{align*}
where the kinetic energy inequality is used in the last step. Sending $\epsilon\to 0$ and $R\to\infty$, by the dominated convergence theorem, we prove that $\textup{Tr}(-\Delta)\mathcal{Q}(-\Delta)=\sum\|\phi_j\|_{\dot{H}^2}^2<\infty$.

Consider the case $d\geq 3$. Suppose that $1\leq s<2$, and choose $\delta>0$ such that $q=\frac{d-\delta}{d-2}$. Then, by \eqref{truncation trick} and the Sobolev and H\"older inequalities,
\begin{align*}
\sum \|P\phi_j\|_{\dot{H}^{s+\delta}}^2&\leq\sum \|(-\Delta+\mu_j)P\phi_j\|_{\dot{H}^{s+\delta-2}}^2=\sum \|P(\rho_{\mathcal{Q}}^{q-1}\phi_j)\|_{\dot{H}^{s+\delta-2}}^2\\
&\lesssim\sum \|\rho_{\mathcal{Q}}^{q-1}\phi_j\|_{L^{\frac{2d}{d+4-2(s+\delta)}}}^2\leq\|\rho_{\mathcal{Q}}\|_{L^{\frac{d}{d-2}}}^{2(q-1)}\sum \|\phi_j\|_{L^{\frac{2d}{d-2s}}}^2\\
&\lesssim \|\rho_{\mathcal{Q}}\|_{L^{\frac{d}{d-2}}}^{2(q-1)}\sum \|\phi_j\|_{\dot{H}^s}^2\lesssim \sum \|\phi_j\|_{\dot{H}^s}^2.
\end{align*}
Thus, by the dominated convergence theorem again, we obtain 
$$\textup{Tr}|\nabla|^{s+\delta}\mathcal{Q}|\nabla|^{s+\delta}\lesssim\textup{Tr}|\nabla|^{s}\mathcal{Q}|\nabla|^{s}.$$
Iterating from $s=1$ finitely many times (but choosing smaller $\delta$ in the last step), we prove that $\textup{Tr}(-\Delta)\mathcal{Q}(-\Delta)<\infty$.
\end{proof}

Next we show the Pohozaev identities.

\begin{proof}[Proof of Theorem \ref{main theorem} (5)]
Taking the inner product of the equation \eqref{EL for all} with $\phi_j$ and then summing in $j$, we obtain the identity
\begin{equation}\label{Po1}
\textup{Tr}\sqrt{-\Delta}\mathcal{Q}\sqrt{-\Delta}-\|\rho_{\mathcal{Q}}\|_{L^q(\mathbb{R}^d)}^q+\sum \mu_j=0.
\end{equation}
Let $\chi$ be a compactly smooth function $\chi$ such that $\chi(x)=\frac{|x|^2}{2}$ for $|x|\leq 1$ and $\chi=0$ for $|x|\geq 2$, and define $\chi_R(x)=R^2\chi(\frac{x}{R})$. Then, taking the inner product of the equation \eqref{EL for all} with $\nabla\chi_R\cdot \nabla\phi_j$ and splitting the three terms in the equation, we write
	\begin{align*}
	0&=-\textup{Re}\sum \langle \nabla\chi_R\cdot \nabla\phi_j, \Delta\phi_j\rangle_{L^2}-\textup{Re}\sum \langle \nabla\chi_R\cdot \nabla\phi_j, \rho_{\mathcal{Q}}^{q-1}\phi_j\rangle_{L^2}\\
	&\quad+\textup{Re}\sum \mu_j\langle \nabla\chi_R\cdot \nabla\phi_j, \phi_j\rangle_{L^2}\\
	&=:I_{R}+{II}_{R}+{III}_{R}.
\end{align*}
Here, Theorem \ref{main theorem} (4) assures summability of the first term $I_R$. For $I_R$, by integration by parts, we write
\begin{align*}
&\sum_{m,n=1}^d\textup{Re}\int_{\mathbb{R}^d} (\partial_{x_m}\partial_{x_n}\chi_R)\partial_{x_m}\phi_j\overline{\partial_{x_n}\phi_j} dx+\sum_{m=1}^d\textup{Re}\int_{\mathbb{R}^d}(\partial_{x_m}\chi_R)\nabla\partial_{x_m}\phi_j\cdot\overline{\nabla\phi_j} dx\\
&=\sum_{m,n=1}^d\textup{Re}\int_{\mathbb{R}^d} (\partial_{x_m}\partial_{x_n}\chi_R)\partial_{x_m}\phi_j\overline{\partial_{x_n}\phi_j} dx+\frac{1}{2}\sum_{m=1}^d\int_{\mathbb{R}^d} (\partial_{x_m}\chi_R)\partial_{x_m}(|\nabla\phi_j|^2) dx\\
&=\sum_{m,n=1}^d\textup{Re}\int_{\mathbb{R}^d} (\partial_{x_m}\partial_{x_n}\chi_R)\partial_{x_m}\phi_j\overline{\partial_{x_n}\phi_j} dx-\frac{1}{2}\int_{\mathbb{R}^d} (\Delta\chi_R)(|\nabla\phi_j|^2) dx.
\end{align*}
Next, summing in $j$ and taking the limit $R\to\infty$ with the dominated convergence theorem, we prove that
$$I_R\to -\frac{d-2}{2}\sum\|\nabla\phi_j\|_{L^2}^2=-\frac{d-2}{2}\textup{Tr}\sqrt{-\Delta}\mathcal{Q}\sqrt{-\Delta}$$
as $R\to \infty$. For $II_R$ and ${III}_{R}$, by integration by parts,
\begin{align*}
{II}_{R}&=-\frac{1}{2q}\int_{\mathbb{R}^d} \nabla\chi_R\cdot\nabla(\rho_{\mathcal{Q}}^q) dx=\frac{1}{2q}\int_{\mathbb{R}^d} (\Delta\chi_R) \rho_{\mathcal{Q}}^qdx\to \frac{d}{2q}\|\rho_{\mathcal{Q}}\|_{L^q}^q,\\
{III}_{R}&=\sum\frac{\mu_j}{2} \int_{\mathbb{R}^d} \nabla\chi_R \cdot\nabla|\phi_j|^2 dx=\sum-\frac{\mu_j}{2} \int_{\mathbb{R}^d} (\Delta\chi_R) |\phi_j|^2 dx\to -\frac{d}{2}\sum \mu_j
\end{align*}
as $R\to \infty$. Collecting all, we get
\begin{equation}\label{Po2}
(d-2)\textup{Tr}\sqrt{-\Delta}\mathcal{Q}\sqrt{-\Delta}-\frac{d}{q}\|\rho_{\mathcal{Q}}\|_{L^q}^q+d\sum_{j=1}^\infty\mu_j=0.
\end{equation}
Finally, solving the system of equations \eqref{Po1} and \eqref{Po2} for $\textup{Tr}\sqrt{-\Delta}\mathcal{Q}\sqrt{-\Delta}$ and $\|\rho_{\mathcal{Q}}\|_{L^q}^q$, we prove the Pohozaev identities \eqref{Po}.
\end{proof}

From the construction, we only have boundedness of $\|\phi_j\|_{L^2}$ (from $\|\mathcal{Q}\|_{\textup{op}}\leq 1$) but no summability a priori-ly. However, if we further assume that $d\geq 3$ and $q>\frac{d^2+2d+4}{d^2}$, we can upgrade summability to be traceable.

\begin{proof}[Proof of Theorem \ref{main theorem} (6)]
Let $V=\rho_{\mathcal{Q}}^{q-1}$. By the assumption ($\frac{d^2+2d+4}{d^2}<q<\frac{d}{d-2}$), we have $\frac{d+2}{d}<\frac{d(q-1)}{2}<\frac{d}{d-2}$. Thus, by the kinetic energy inequality \eqref{eq:LT}, the potential function $-V$ is contained in $L^{\frac{d}{2}}$. Thus, the number $J_-$ of negative eigenvalues for the Schr\"odinger operator $(-\Delta-V)$ is finite.

Choose $\delta>0$ such that $q=\frac{d^2+(d+2)(2+\delta)}{d^2}$. Then, we have the bound $\|V\|_{L^{\frac{d}{2+\delta}}}=\|\rho_{\mathcal{Q}}\|_{L^{\frac{d+2}{d}}}^{q-1}<\infty$. Therefore, by the equation \eqref{EL for all} and the Sobolev and H\"older inequalities, 
	\begin{align*}
	\| |\nabla|^{s-\delta}\phi_j\|_{L^2}&=\||\nabla|^{s-\delta}(-\Delta+\mu_j)^{-1}(V\phi_j)\|_{L^2}\\
	&\lesssim\|V\phi_j\|_{L^{\frac{2d}{d+4-2s+2\delta}}} \leq \|V\|_{L^{\frac{d}{2+\delta}}}\N{\phi_j}_{L^{\frac{2d}{d-2s}}}\\
	&\lesssim\N{|\nabla|^s \phi_j}_{L^2},
	\end{align*}
where the implicit constant does not depends on $j$ on $s\in[\delta,1]$. Iterating this inequality from $s=1$, we obtain $\N{\phi_j}_{L^2}\lesssim \N{\nabla\phi_j}_{L^2}$. Thus, squaring and summing in $j$, we conclude that $\Tr(\mathcal{Q})=\sum\N{\phi_j}_{L^2}^2\lesssim \sum \N{\nabla\phi_j}_{L^2}^2=\|\mathcal{Q}\|_{\dot{\frak{H}}^1}<\infty$.
\end{proof}

\vspace{10pt}

\section{Global versus blow-up dichotomy: Proof of Theorem \ref{Dichotomy}}
Finally, coming back to the PDE problem \eqref{CNLS} (equivalently \eqref{CNLS'}), we prove our main result (Theorem \ref{Dichotomy}). 

\subsection{Local theory}\label{sec: local theory}
To begin with, we prepare a local well-posedness theory defining suitable function spaces. For a sequences of functions $\vec{\Phi}=\{\phi_j\}_{j=1}^\infty$, which are not necessarily mutually orthogonal, we define the Sobolev norm by 
$$\|\vec{\Phi}\|_{\vec{H}^1}:=\left\{\sum_{j=1}^\infty \|\phi_j\|_{H^1(\mathbb{R}^3)}^2\right\}^{1/2}.$$
Note that if $\phi_j$'s are mutually orthogonal, then by identification
$$\textup{a sequence }\vec{\Phi}=\{\phi_j\}_{j=1}^\infty \leftrightarrow\textup{an non-negative operator }\gamma=\sum_{j=1}^\infty |\phi_j\rangle\langle \phi_j|,$$
we have $\|\vec{\Phi}\|_{\vec{H}^1}^2=\textup{Tr}\sqrt{1-\Delta}\gamma\sqrt{1-\Delta}=\|\gamma\|_{\mathfrak{H}^1}$. For a time-dependent sequence $\vec{\Phi}(t)=\{\phi_j(t)\}_{j=1}^\infty$ and a time interval $I\subset\mathbb{R}$, we define the vector-valued Strichartz norm by
$$\|\vec{\Phi}(t)\|_{\vec{S}^1(I)}:=\left\{\sum_{j=1}^\infty \|\phi_j(t)\|_{S^1(I)}^2\right\}^{1/2},$$
where 
$$\|u\|_{S^1(I)}:=\|u\|_{C_t(I;H_x^1(\mathbb{R}^3))\cap L_t^2(I;W_x^{1, 6}(\mathbb{R}^3))}.$$
Let $\vec{H}^1$ (or $\vec{S}^1(I)$, respectively) be the collection of sequences having finite $\vec{H}^1$-norm (or the $\vec{S}^1(I)$-norm, respectively), which is a complete metric space.

We establish local well-posedness of the system of NLS \eqref{CNLS} without assuming the orthogonality condition.
\begin{proposition}[Local well-posedness for CNLS \eqref{CNLS}]\label{LWP}
\ 
\begin{enumerate}
\item For every initial data $\vec{\Phi}_0\in\vec{H}^1$, there exist a time interval $I\subset\mathbb{R}$ and a unique strong solution $\vec{\Phi}(t)=\{\phi_j(t)\}_{j=1}^\infty\in \vec{S}^1(I)$ to CNLS \eqref{CNLS}, i.e.,
$$\phi_j(t)=e^{it\Delta}\phi_{j,0}-i\int_0^t e^{i(t-s)\Delta}(\rho\phi_j)(s)ds\quad\forall j\in\mathbb{N},$$
where $\rho=\sum_{j=1}^\infty|\phi_j|^2$.
\item Let $(-T_{min},T_{max})$ be the maximal interval of existence. If $T_{min}$ (or $T_{max}<\infty$, respectively), then $\|\vec{\Phi}(t)\|_{\vec{H}^1}\to\infty$ as $t\nearrow T_{max}$ (or as $t\searrow -T_{min}$, respectively).
\item For any closed interval $I\subset(-T_{min},T_{max})$, the data-to-solution map $\vec{\Phi}_0\mapsto \vec{\Phi}(t)$ is continuous from $\vec{H}^1$ to $\vec{S}^1(I)$.
\end{enumerate}
\end{proposition}

\begin{remark}
For CNLS \eqref{CNLS'}, Proposition \ref{LWP} corresponds to local well-posedness in $\mathfrak{H}^1$. It could be possible to upgrade the local theory in a weaker summable space using Strichartz estimates in \cite{FLLS14, FS, CHP1, CHP2, BHLNS}. However, in this paper, we do not pursue to find the optimal (or least) summable space. 
\end{remark}

The main tool to prove local well-posedness is the following standard Strichartz estimates.
\begin{lemma}[Strichartz estimates]\label{Strichartz estimates}
\begin{align*}
\|e^{it\Delta}u_0\|_{C_t(I;L_x^2(\mathbb{R}^3))\cap L_t^2(I;L_x^{6}(\mathbb{R}^3))}&\lesssim \|u_0\|_{L^2(\mathbb{R}^3)},\\
\left\|\int_0^t e^{i(t-s)\Delta}F(s)ds\right\|_{C_t(I;L_x^2(\mathbb{R}^3))\cap L_t^2(I;L_x^{6}(\mathbb{R}^3))}&\lesssim \|F\|_{L_t^1(I;L_x^2(\mathbb{R}^3))\cup L_t^2(I;L_x^{6/5}(\mathbb{R}^3))}.
\end{align*}
\end{lemma}

\begin{proof}[Proof of Proposition \ref{LWP}]
Given initial data $\vec{\Phi}_0=\{\phi_{j,0}\}_{j=1}^\infty$, we define the nonlinear mapping $\vec{\Gamma}$ by
$$\vec{\Gamma}\left(\vec{\Phi}(t)\right)=\vec{\Gamma}_{\vec{\Phi}_0}\left(\vec{\Phi}(t)\right):=\Big\{\Gamma_j(\phi_j)(t)\Big\}_{j=1}^\infty,$$
where 
$$\Gamma_j\left(\phi_j(t)\right):=e^{it\Delta}\phi_{j,0}-i\int_0^t e^{i(t-s)\Delta}(\rho\phi_j)(s)ds.$$
Let $T>0$ be a small number to be specified later, and let $I=[-T,T]$. First, applying Strichartz estimates (Lemma \ref{Strichartz estimates}) to $\Gamma_j(\phi_j(t))$, we get
\begin{align*}
\|\Gamma_j(\phi_j(t))\|_{S^1(I)}&\lesssim \|\phi_{j,0}\|_{H^1(\mathbb{R}^3)}+\|(\rho\phi_j)(t)\|_{L_{t}^2(I;W_x^{1,6/5}(\mathbb{R}^3))}\\
&\lesssim \|\phi_{j,0}\|_{H^1(\mathbb{R}^3)}+T^{1/2}\|(\rho \phi_j)(t)\|_{C_{t}(I;W_x^{1,6/5}(\mathbb{R}^3))}.
\end{align*}
By the Leibniz rule, the H\"older inequality, the Sobolev inequality and Young's inequality, 
\begin{align*}
\|\rho \phi_j\|_{W^{1,6/5}(\mathbb{R}^3)}&\leq\|\rho \phi_j\|_{L^{6/5}(\mathbb{R}^3)}+\|(\nabla\rho) \phi_j\|_{L^{6/5}(\mathbb{R}^3)}+\|\rho(\nabla\phi_j)\|_{L^{6/5}(\mathbb{R}^3)}\\
&\leq\|\rho\|_{L^{3}(\mathbb{R}^3)}\|\phi_j\|_{L^{2}(\mathbb{R}^3)}+\|\nabla\rho\|_{L^{3/2}(\mathbb{R}^3)}\|\phi_j\|_{L^{6}(\mathbb{R}^3)}+\|\rho\|_{L^3(\mathbb{R}^3)}\|\nabla\phi_j\|_{L^2(\mathbb{R}^3)}\\
&\lesssim\|\nabla\rho\|_{L^{3/2}(\mathbb{R}^3)}\|\phi_j\|_{H^1(\mathbb{R}^3)},
\end{align*}
which is, by the definition of the density function, bounded by 
\begin{align*}
\left\{\sum\|\nabla|\phi_k|^2\|_{L^{3/2}(\mathbb{R}^3)}\right\}\|\phi_j\|_{H^1(\mathbb{R}^3)}&\leq2\left\{\sum \|\nabla\phi_k\|_{L^{2}(\mathbb{R}^3)}\|\phi_k\|_{L^{6}(\mathbb{R}^3)}\right\}\|\phi_j\|_{H^1(\mathbb{R}^3)}\\
&\lesssim \left\{\sum\|\phi_k\|_{H^1(\mathbb{R}^3)}^2\right\}\|\phi_j\|_{H^1(\mathbb{R}^3)}.
\end{align*}
Thus, we have
$$\|\Gamma_j(\phi_j(t))\|_{S^1(I)}\lesssim \|\phi_{j,0}\|_{H^1(\mathbb{R}^3)}+T^{1/2} \|\vec{\Phi}(t)\|_{\vec{S}^1(I)}^2\|\phi_j(t)\|_{S^1(I)}.$$
Taking the $\ell^2$-sum over $j$, we prove that 
$$\left\|\vec{\Gamma}\left(\vec{\Phi}(t)\right)\right\|_{\vec{S}^1(I)}\leq C\|\vec{\Phi}_0\|_{\vec{H}^1}+CT^{1/2}\|\vec{\Phi}(t)\|_{\vec{S}^1(I)}^3.$$
By the same way, one can show that 
\begin{align*}
&\left\|\vec{\Gamma}\left(\vec{\Phi}(t)\right)-\vec{\Gamma}\left(\vec{\Psi}(t)\right)\right\|_{\vec{S}^1(I)}\leq CT^{1/2}\left\{\|\vec{\Phi}(t)\|_{\vec{S}^1(I)}^2+\|\vec{\Psi}(t)\|_{\vec{S}^1(I)}^2\right\}\|\vec{\Phi}(t)-\vec{\Psi}(t)\|_{\vec{S}^1(I)}.
\end{align*}
Therefore, taking sufficiently small $T>0$ such that
$$CT^{1/2}\left(2C\|\vec{\Phi}_0\|_{\vec{H}^1}\right)^2\leq\frac{1}{4},$$
we prove that $\vec{\Gamma}$ is a contraction on the set
$$\left\{\vec{\Phi}(t)\in \vec{S}^1(I):\ \|\vec{\Phi}(t)\|_{\vec{S}^1(I)}\leq 2C\|\vec{\Phi}_0\|_{\vec{H}^1}\right\}.$$
Then, the proposition follows as a consequence of the contraction mapping principle.
\end{proof}

Next, we show conservation of orthogonality \eqref{orthogonality of waves} and the energy \eqref{energy}, which are equivalent to the conservation laws \eqref{conservation laws} in the operator form.
\begin{proposition}[Conservation laws]\label{prop: conservation laws}
Let $\vec{\Phi}(t)=\{\phi_j(t)\}_{j=1}^\infty\in C_{t}(I;\vec{H}^1)$ be the solution to \eqref{CNLS}, constructed in Proposition \ref{LWP}. Then, the orthogonality $\langle \phi_j, \phi_k\rangle_{L^2(\mathbb{R}^3)}=\delta_{jk}\|\phi_j\|_{L^2(\mathbb{R}^3)}^2$ and the energy $\mathcal{E}(\vec{\Phi})$ are conserved on the time interval $I$. 
\end{proposition}

\begin{proof}
We prove the conservation laws for smooth solutions substituting $\partial_t\phi_j$ by the equation \eqref{CNLS} and doing integration by parts, 
\begin{align*}
\frac{d}{dt}\langle \phi_j, \phi_k\rangle_{L^2(\mathbb{R}^3)}&=\langle \partial_t\phi_j, \phi_k\rangle_{L^2(\mathbb{R}^3)}+\langle \phi_j, \partial_t\phi_k\rangle_{L^2(\mathbb{R}^3)}\\
&=i\langle (\Delta+\rho)\phi_j, \phi_k\rangle_{L^2(\mathbb{R}^3)}-i\langle \phi_j, (\Delta+\rho)\phi_k\rangle_{L^2(\mathbb{R}^3)}=0
\end{align*}
and
\begin{align*}
\frac{d}{dt}\mathcal{E}(\vec{\Phi})&=\textup{Re}\sum\int_{\mathbb{R}^3} \nabla\phi_j\cdot \overline{\nabla\partial_t\phi_j}-\rho \phi_j\overline{\partial_t\phi_j} dx\\
&=-\textup{Re}\sum\int_{\mathbb{R}^3} (\Delta\phi_j+\rho \phi_j)\overline{\partial_t\phi_j} dx\\
&=-\textup{Im}\sum\int_{\mathbb{R}^3}|\Delta\phi_j+\rho \phi_j|^2 dx=0.
\end{align*}
The conservation laws for $\vec{H}^1$-strong solutions can be justified by the standard persistence of regularity argument. 
\end{proof}

By the (total) mass conservation, the blow-up criterion in Proposition \ref{LWP} (2) can be rewritten as follows.

\begin{corollary}[Blow-up criterion]\label{blow-up criterion}
Let $\vec{\Phi}(t)\in C_{t}(I;\vec{H}^1)$ be the solution to \eqref{CNLS}, constructed in Proposition \ref{LWP}, and let $(-T_{min},T_{max})$ be its maximal interval of existence. If $T_{\textup{max}}$ (or $T_{\textup{min}}$, respectively) is finite, then  $\sum_{j=1}^\infty\|\nabla\phi_j(t)\|_{L^2(\mathbb{R}^3)}^2\to \infty$ as $t\nearrow T_{\textup{max}}$ (or as $t\searrow -T_{\textup{min}}$, respectively).
\end{corollary}

The following virial identity is the key ingredient to prove the blow-up part in our main theorem. For notational convenience, we here state it in the operator form.

\begin{lemma}[Virial indentity]\label{virial identity}
If $\gamma(t)$ is a non-negative solution to CNLS \eqref{CNLS'} having finite variance $\int_{\R^3}|x|^2\rho_{\gamma(t)} dx<\infty$, then
\begin{equation}\label{eq:Virial'}
\partial_t^2\int_{\R^3}|x|^2\rho_{\gamma(t)} dx=8\|\gamma\|_{\dot{\mathfrak{H}}^1}-6\N{\rho_\gamma}_{L^2}^2.
\end{equation}
\end{lemma}

\begin{proof}
By the equation \eqref{CNLS},
$$\partial_t|\phi_j|^2=2\textup{Re}\left(\partial_t\phi_j\overline{\phi_j}\right)=2\textup{Re}\left(i(\Delta\phi_j+\rho\phi_j)\overline{\phi_j}\right)=-2\textup{Im}\left(\Delta\phi_j\overline{\phi_j}\right)=-2\nabla\cdot \textup{Im}(\nabla\phi_j\overline{\phi_j}).$$
Thus, by integration by parts, we obtain
$$\partial_t\int_{\R^3}|x|^2\rho_{\g(t)} dx=\sum\int_{\R^3}|x|^2\partial_t|\phi_j|^2dx=4\sum\int_{\R^3}x\cdot\textup{Im}(\nabla\phi_j\overline{\phi_j}) dx.$$
Similarly, we have
\begin{align*}
\partial_t \left\{\int_{\R^3}x\cdot\textup{Im}(\nabla\phi_j\overline{\phi_j}) dx\right\}&=\Ima\int_{\R^3}(x\cdot\nabla\partial_t\phi_j)\overline{\phi_j}+(x\cdot\nabla\phi_j)\overline{\partial_t\phi_j}dx\\
&=-\Ima\int_{\R^3}3(\partial_t\phi_j)\overline{\phi_j}dx+2\overline{(x\cdot\nabla\phi_j)}\partial_t\phi_jdx\\
&=-\Rea\int_{\R^3}3(\Delta\phi_j+\rho\phi_j)\overline{\phi_j}+2\overline{(x\cdot\nabla\phi_j)}(\Delta\phi_j+\rho\phi_j)dx\\
&=\int_{\R^3}2|\nabla\phi_j|^2-3\rho_\g|\phi_j|^2+x\rho_\g \cdot\nabla(|\phi_j|^2)dx.
\end{align*}
Summing over $j$ and integrating by parts again, we prove that 
$$\partial_{t}^2\int_{\R^3}|x|^2\rho_{\g(t)}dx=8\int_{\R^3}|\nabla\phi_j|^2dx-6\int_{\mathbb{R}^3}\rho_\g^2dx=8\|\gamma\|_{\dot{\mathfrak{H}}^1}-6\|\rho_\gamma\|_{L^2}^2.$$
\end{proof}

\subsection{Proof of Theorem \ref{Dichotomy}}

First, we will prove the upper bound versus lower bound dichotomy employing the sharp kinetic energy inequality \eqref{sharp kinetic energy inequality}, the Pohozaev identities
\begin{equation}\label{Po'}
\|\mathcal{Q}\|_{\dot{\mathfrak{H}}^1}=3\sum_{j=1}^{J^-} \mu_j;\quad\|\rho_{\mathcal{Q}}\|_{L^2(\mathbb{R}^3)}^2=4\sum_{j=1}^{J^-} \mu_j.
\end{equation}
(from \eqref{Po} in Theorem \ref{main theorem}) and the virial identity (Lemma \ref{virial identity}). In one case in the dichotomy, global well-posedness immediately follows from the upper bound. On the other hand, in the other case in the dichotomy, we can show finite time blow-up by the standard Glassey argument \cite{Glassey} assuming finite variance in addition. 

Given $\gamma=\sum_{j=1}^\infty |\phi_j\rangle\langle\phi_j|$ with $\|\gamma\|_{\textup{op}}=1$, we denote
$$\alpha:=\|\gamma\|_{\dot{\mathfrak{H}}^1}^{1/2}.$$
Then, by the kinetic energy inequality \eqref{sharp kinetic energy inequality} with normalization $\|\mathcal{Q}\|_{\textup{op}}=1$, we have
\begin{equation}\label{eq:PLTf}
\begin{aligned}
\E(\g)&\geq\frac{1}{2}\|\gamma\|_{\dot{\mathfrak{H}}^1}-\frac{\|\rho_{\mathcal{Q}}\|_{L^2(\mathbb{R}^3)}^2}{4\|\mathcal{Q}\|_{\textup{op}}^{1/2}\|\mathcal{Q}\|_{\dot{\mathfrak{H}}^1}^{3/2}}\|\gamma\|_{\textup{op}}^{1/2}\|\gamma\|_{\dot{\mathfrak{H}}^1}^{3/2}\\
&=\frac{1}{2}\|\gamma\|_{\dot{\mathfrak{H}}^1}-\frac{1}{3\|\mathcal{Q}\|_{\dot{\mathfrak{H}}^1}^{1/2}}\|\gamma\|_{\dot{\mathfrak{H}}^1}^{3/2}=f\left(\alpha\right),
\end{aligned}
\end{equation}
where the Pohozaev identities \eqref{Po'} are used in the first identity, and
$$f(\alpha):=\frac{1}{2}\alpha^2-\frac{1}{3\|\mathcal{Q}\|_{\dot{\mathfrak{H}}^1}^{1/2}}\alpha^3.$$
Observe that on $[0,\infty)$, the function $f(\alpha)$ has a local minimum at 0 and the global maximum at $\alpha_*=\|\mathcal{Q}\|_{\dot{\mathfrak{H}}^1}^{1/2}$. Moreover, by the Pohozaev identities \eqref{Po'} again, 
$$f(\alpha_*)=\E(\mathcal{Q}).$$
See Figure \ref{fig1}.

Let $\gamma(t)\in\mathfrak{H}^1$ be a solution to \eqref{CNLS'}. Then, by \eqref{eq:PLTf}, the conservation laws and the assumption, we have
$$f(\alpha(t))\leq \E(\g(t))=\E(\g_0)<\E(\mathcal{Q})=f(\alpha_*).$$
Thus, by continuity of the function $\alpha(t)$, it follows that either
\begin{equation}\label{dichotomy case 1}
\alpha(t)=\|\gamma(t)\|_{\dot{\mathfrak{H}}^1}^{1/2}\leq\alpha_*=\|\mathcal{Q}\|_{\dot{\mathfrak{H}}^1}^{1/2}
\end{equation}
on the interval $I_{max}$, or
\begin{equation}\label{dichotomy case 2}
\alpha(t)=\|\gamma(t)\|_{\dot{\mathfrak{H}}^1}^{1/2}\geq \alpha_*=\|\mathcal{Q}\|_{\dot{\mathfrak{H}}^1}^{1/2}
\end{equation}
on the interval $I_{max}$.

In the former case \eqref{dichotomy case 1}, global well-posedness follows from the uniform bound on $\|\gamma(t)\|_{\dot{\mathfrak{H}}^1}$ and Corollary \ref{blow-up criterion}. In the latter case \eqref{dichotomy case 2}, if we further assume that 
$$\int_{\R^3}|x|^2\rho_{\g_0}dx<\infty,$$
then by the virial identity (Lemma \ref{virial identity}) and the energy conservation, we get 
$$\partial_{t}^2\int_{\R^3}|x|^2\rho_{\g}dx=8\|\gamma\|_{\dot{\mathfrak{H}}^1}-6\N{\rho_\gamma}_{L^2}^2=24\mathcal{E}(\gamma)-4\|\gamma\|_{\dot{\mathfrak{H}}^1}=24\mathcal{E}(\gamma_0)-4\|\gamma\|_{\dot{\mathfrak{H}}^1}.$$
Consequently, by \eqref{dichotomy case 2} and the Pohozaev identities \eqref{Po'}, 
$$\partial_{t}^2\int_{\R^3}|x|^2\rho_{\g}dx\leq24\mathcal{E}(\gamma_0)-4\|\mathcal{Q}\|_{\dot{\mathfrak{H}}^1}=24\left(\mathcal{E}(\gamma_0)-\mathcal{E}(\mathcal{Q})\right)<0.$$
Since the second derivative of the variance is bounded by a negative number uniformly in time, the maximal interval of existence $I_{max}$ must be finite. In other words, $\gamma(t)$ blows up in finite time, both in forward and backward.

\end{document}